\documentclass[a4paper]{amsart}
\usepackage{txfonts, amsmath,amstext,amsthm,amscd,amsopn,verbatim,amssymb, amsfonts}
\usepackage{fullpage}
\usepackage[all]{xy}

\usepackage[bbgreekl]{mathbbol}
\usepackage{tikz}
\usetikzlibrary{matrix}
	\usetikzlibrary{positioning, calc}
\usetikzlibrary{shapes}
\usetikzlibrary{arrows}  
\usetikzlibrary{calc,3d}
\usetikzlibrary{decorations,decorations.pathmorphing}
\usetikzlibrary{through}
\tikzset{ext/.style={circle, draw,inner sep=1pt},int/.style={circle,draw,fill,inner sep=1pt},nil/.style={inner sep=1pt}}
\tikzset{exte/.style={circle, draw,inner sep=3pt},inte/.style={circle,draw,fill,inner sep=3pt}}
\tikzset{diagram/.style={matrix of math nodes, row sep=3em, column sep=2.5em, text height=1.5ex, text depth=0.25ex}}
\tikzset{diagram2/.style={matrix of math nodes, row sep=0.5em, column sep=0.5em, text height=1.5ex, text depth=0.25ex}}
\theoremstyle{plain}
  \newtheorem{thm}{Theorem}
  \newtheorem{defi}{Definition}[section]
  
  \newtheorem{prop}[defi]{Proposition}
  
  \newtheorem{cor}{Corollary}
  \newtheorem{lemma}[defi]{Lemma}
\theoremstyle{definition}
  \newtheorem{ex}[defi]{Example}
  \newtheorem{rem}[defi]{Remark}

\newcommand{\alg}[1]{\mathfrak{{#1}}}

\newcommand{\eref}[1]{\eqref{#1}} 

\newcommand{\Hom}{\mathrm{Hom}}

\newcommand{\R}{{\mathbb{R}}}
\newcommand{\Z}{{\mathbb{Z}}}

\newcommand{\Q}{{\mathbb{Q}}}
\newcommand{\U}{{\mathcal{U}}}
\newcommand{\X}{{\mathcal{X}}}
\renewcommand{\L}{{{L}}}

\newcommand{\op}{\mathcal}

\newcommand{\Lie}{\mathsf{Lie}}

\newcommand{\coker}{\mathrm{coker}}

\newcommand{\Ass}{\mathsf{Assoc}}
\newcommand{\Com}{\mathsf{Com}}

\newcommand{\bbS}{\mathbb{S}}

\newcommand{\bpm}{\begin{pmatrix}}
\newcommand{\epm}{\end{pmatrix}}

\newcommand{\mU}{\mathcal{U}}

\DeclareMathOperator{\End}{End}

\newcommand{\gr}{\mathit{gr}}
\newcommand{\GL}{\mathrm{GL}}

\newcommand{\Out}{\mathrm{Out}}
\newcommand{\Aut}{\mathrm{Aut}}
\newcommand{\Fin}{\mathrm{Fin}}

\newcommand{\dimens}{\mathop{dim}}

\newcommand{\FreeLie}{\mathrm{FreeLie}}

\newcommand{\diag}{{\mathrm{diag}}\,}
\newcommand{\Tot}{{\mathrm{Tot}}}
\newcommand{\CH}{{\mathrm{CH}}}
\newcommand{\CHH}{{\mathcal{CH}}}
\newcommand{\Mod}{\mathrm{mod}}
\newcommand{\TOP}{{\mathrm{Top}}}

\newcommand{\coLie}{\mathsf{coLie}}

\newcommand{\hRmod}{\operatorname{hRmod}}
\newcommand{\Rmod}{\operatorname{Rmod}}

\newcommand{\calA}{{\mathcal A}}
\newcommand{\calL}{{\mathcal L}}

\newcommand{\Harr}{{ \mathrm{Harr} }}

\newcommand{\Sh}{{\mathit{Sh}}}

\newcommand{\Chev}{\mathbf{C}} 
\newcommand{\ChevMod}{\Chev_{\rm mod}}
\newcommand{\Ha}{\calL} 
\newcommand{\HaMod}{\calL_{\rm mod}} 
 
\newcommand{\hChev}{\hat{\Chev}} 
\newcommand{\hChevMod}{\hChev_{\rm mod}}
\newcommand{\hHa}{\hat \calL} 

\begin{document}
\title{Hochschild-Pirashvili homology on suspensions and\\ representations of $\Out(F_n)$}

\author{Victor Turchin}
\address{Department of Mathematics\\
  Kansas State University\\
  138 Cardwell Hall\\
  Manhatan, KS 66506, USA}
  \email{turchin@ksu.edu}
\author{Thomas Willwacher}
\address{Institute of Mathematics \\ University of Zurich \\  
Winterthurerstrasse 190 \\
8057 Zurich, Switzerland}
\email{thomas.willwacher@math.uzh.ch}

\thanks{V.T. acknowledges partial support by the MPIM, Bonn, and the IHES . T.W. acknowledges partial support by the Swiss National Science Foundation (grant 200021\_150012) and the SwissMap NCCR, funded by the Swiss National Science Foundation}

 \subjclass[2010]{55N99;19D55;13D03}

\begin{abstract}
We show that the Hochschild-Pirashvili homology on any suspension admits the so called Hodge splitting.
For a map between suspensions $f\colon \Sigma Y\to \Sigma Z$, the induced map in the 
Hochschild-Pirashvili homology preserves this splitting if $f$ is a suspension. If $f$ is not a suspension, we 
show that the splitting is preserved only as a filtration.   
%
%
As a special case, we obtain that the  Hochschild-Pirashvili homology on wedges of circles produces new representations  of $\Out(F_n)$ that do not factor in general through $\GL(n,\Z)$. The obtained representations are naturally filtered in such a way that the  action on  the graded quotients  does factor through $\GL(n,\Z)$.


%
\end{abstract}

\maketitle

\setcounter{section}{-1}

\section{Introduction}\label{s:HPH}

The higher Hochschild homology is a bifunctor introduced  by T.~Pirashvili in~\cite{pirash00} that 
to a topological space (simplicial set) and a (co)commutative (co)algebra assigns a graded vector space. 
Informally speaking this functor is a way to \lq\lq integrate\rq\rq{} a (co)algebra over a given space.  Specialized to a circle the result is the usual Hochschild homology. 
The precursor to the higher Hochschild homology was the discovery of the Hodge splitting in the
usual Hochschild homology of a commutative algebra~\cite{GerstSchack,Loday}. Indeed, the most surprising and perhaps the motivating result for T.~Pirashvili to write his seminal work~\cite{pirash00} was 
the striking fact that the higher Hochschild homology on a sphere of any positive dimension also 
admits the Hodge splitting and moreover the terms of the splitting up to a regrading depend only on the parity of the dimension of the sphere. With this excuse to be born, the higher Hochschild homology is nowadays a widely used tool that has various applications including 
 the string topology and more generally  the study of mapping  and   embedding spaces~\cite{pirash00,AroneTur1,AroneTur2,GTZ0,Patras,PatrasThomas,Song,SongTur}.  It also has very interesting and deep generalizations such as the topological higher Hochschild homology~\cite{CarlDougDun,Schl} and factorization homology~\cite{AyFr,Ginot,GTZ,Lurie}. 

In our work we study the very nature of the Hodge splitting. In particular we show that it always takes place for suspensions.   
 Moreover, it will be clear from the construction that only suspensions and spaces rationally homology equivalent to them have this property. 
For any suspension $\Sigma Y$,  the terms of the splitting depend in some polynomial way on 
$\tilde H_*\Sigma Y$, which in particular explains Pirashvili\rq{}s result for spheres. We also show that if a map $f\colon \Sigma Y\to\Sigma Z$ is a suspension, than the
induced map in the Hochschild-Pirashvili homology preserves the splitting and is determined by the map
 $f_*\colon \tilde H_*\Sigma Y\to \tilde H_*\Sigma Z$. In case $f$ is not a suspension, the Hodge splitting is preserved only as a filtration. We explain how the induced map between different layers is computed from
 the rational homotopy type of $f$.
 
 We treat more carefully the case of wedges of circles and discover certain representations of the group $\Out(F_n)$ 
 of outer automorphisms of a free group\footnote{These representations appear as application to the hairy 
 graph-homology computations in the study of  the spaces of long embeddings, higher dimensional string links, and the deformation theory of the little discs operads~\cite{AroneTur2,SongTur,Turchin1,TW,TW2}.} that have the smallest known dimension among those that don\rq{}t factor through $\GL(n,\Z)$. 
  

\subsection*{Notation}
We work over rational numbers $\Q$ unless otherwise stated. All vector spaces are assumed to be vector spaces over $\Q$. Graded vector spaces are vector spaces with a $\mathbb{Z}$-grading, and we abbreviate the phrase ``differential graded'' by dg as usual.
We generally use  homological conventions, i.e., the differentials will have degree $-1$. We denote by 
$gVect$ and $dgVect$ the category of graded vector spaces and the category of chain complexes respectively. For a chain complex or a graded vector space $C$ we denote by $C[k]$ its $k$-th desuspension.

We use freely the language of operads. A good introduction into the subject can be found in the textbook \cite{lodayval}, whose conventions we mostly follow. We use the notation $\op P\{k\}$ for the $k$-fold operadic suspension. 
The operads governing commutative, associative and Lie algebras are denoted by $\Com$, $\Ass$, and $\Lie$ respectively. By $\Com_+$ we denote the commutative non-unital operad and by $\coLie$ the cooperad dual to $\Lie$. 

For a category ${\mathcal C}$, we denote by ${\mathrm{mod}}{-}{\mathcal C}$ the category of cofunctors
${\mathcal C}^{op}\to dgVect$ to chain complexes. The objects of ${\mathrm {mod}}{-}{\mathcal C}$ will be called {\it right ${\mathcal C}$-modules}. In the following section, ${\mathcal C}$ is either the category $\Gamma$ of finite pointed sets or the category $\Fin$ of finite sets. Abusing notation we denote the set $\{1,\ldots,k\}$ by $k$ and the set $\{*,1,\ldots,k\}$ based at $*$ by $k_*$. We will consider the following examples of right $\Gamma$ and $\Fin$-modules:
\begin{itemize}
\item For $X$ some topological space we can consider the $\Fin$-module sending a finite set $S$ to the singular chains on the mapping space $C_*(X^S)$. We denote this $\Fin$-module by $C_*(X^\bullet)$.
\item Similarly, to a basepointed space $X_*$ we assign a $\Gamma$-module $C_*(X_*^\bullet)$ sending a pointed set $S_*$ to $C_*(X_*^{S_*})$, where now $X_*^{S_*}$ is supposed to be the space of pointed maps.
\item To a cocommutative coalgebra $C$ we assign the $\Fin$-module sending the finite set $S$ to the tensor product $C^{\otimes S}\cong \bigotimes_{s\in S} C$. We denote this $\Fin$-module by $C^{\otimes\bullet}$. If not otherwise stated we assume that $C$ is non-negatively graded and simply connected.
\item If in addition $M$ is a $C$-comodule (e.g., $M=C$) one can construct a $\Gamma$-module $M\otimes C^{\otimes\bullet}$ such that $S_*\mapsto M\otimes \bigotimes_{s\in S_*\setminus\{*\}} C$. 
\item Dually, if $M$ is a module over a commutative algebra $A$, then $M\otimes A^{\otimes\bullet}$ 
is a {\it left} $\Gamma$-module, and its objectwise dual $\left(M\otimes A^{\otimes\bullet}\right)^\vee$
is a right $\Gamma$-module.
\end{itemize}

A topological space is said of finite type if all its homology groups are finitely generated in every degree. 

Two spaces are said rationaly homology equivalent if there is a zigzag of maps between them, such that 
its every map induces an isomorphism in rational homology.

The completed tensor product is denoted by $\hat \otimes$.

\subsection*{Main results}

In the paper for simplicity of exposition we stick to the contravariant Hochschild-Pirashvili homology that is to the one assigned to 
right $\Fin$ and $\Gamma$ modules.  One should mention however that all the results can be easily adjusted to the covariant case as well.

There are two ways to define the higher Hochschild homology. In the first combinatorial way, 
for a  space $X$ (respectively pointed space $X_*$) obtained as a realization of a (pointed) finite simplicial set
$\X_\bullet\colon\Delta^{op}\to \Fin$ (respectively $\X_\bullet\colon \Delta^{op}\to\Gamma$), the higher Hochschild homology $HH^X(\L)$ (respectively $HH^{X_*}(\L_*)$) can be computed as the homology of the totalization of the cosimplicial chain complex $\L\circ \X\colon\Delta\to dgVect$ (respectively 
 $\L_*\circ \X_*\colon\Delta\to dgVect$).~\footnote{This definition can also be adjusted to realizations of any simplicial sets non-necessarily finite by using the right Kan extention of $\L$ (respectively $\L_*$) to
 the category of all  (pointed) sets~\cite{pirash00}.}

In another definition,  for a right $\Fin$-module $\L$ (respectively right $\Gamma$-module $\L_*$) and a topological space $X$ (respectively pointed space $X_*$), the {\it  higher Hochschild homology} that we also call {\it  Hochschild-Pirashvili homology} $HH^X(\L)$ (respectively $HH^{X_*}(\L_*)$) is the homology of the complex of homotopy natural transformations $C_*(X^\bullet) \to \L$ (respectively $C_*(X_*^\bullet) \to \L_*$)~\cite{pirash00,GTZ}. 

The fact that the two definitions are equivalent  is implicitly
 shown  in the proof of~\cite[Theorem~2.4]{pirash00} by Pirashvili, see also~\cite[Proof of Proposition~4]{GTZ} and~\cite[Proposition~3.4]{Song}. 

In case $\L=C^{\otimes\bullet}$ (respectively $\L=M\otimes C^{\otimes\bullet}$), we denote the higher Hochschild homology as 
$HH^{X}(C)$ (respectively $HH^{X_*}(C,M)$).\footnote{This particular case of higher Hochschild homology  is also called topological factorisation (or chiral) cohomology, see for example~\cite{AyFr,GTZ}.}
 
 In our paper the combinatorial definition will be used only for wedges of circles as we want to treat this case more explicitly. Later in the paper we  show that 
 for wedges of circles the first and the second  definitions  produce   identical complexes. 
 
Any map $f:X \to Y$ (respectively basepoint preserving map $X_*\to Y_*$) induces a map $f^*: HH^{Y}(\L)\to HH^X(\L)$ (respectively $HH^{Y_*}(\L_*)\to HH^{X_*}(\L_*)$).
Two homotopic maps (respectively basepoint homotopic maps) induce the same map in higher Hochschild homology.
It is also clear from the (first) definition that in case $f$ is a rational homology equivalence, then the induced map $f^*$ is an isomorphism.   
   One has a functor $u\colon\Gamma\to\Fin$ that forgets the basepoint.  If $X=X_*$ and $\L_*=\L\circ u$, then
\begin{equation}\label{eq:base}
HH^{X}(\L)=HH^{X_*}(\L_*).
\end{equation}

In  case we take $X$ and $X_*$ to be a wedge of $n$ circles $\vee_n S^1$,
the automorphism group $\Aut(F_n)$ acts on $\vee_n S^1$ up to homotopy by basepoint preserving maps and hence we obtain a representation of $\Aut(F_n)$ on $HH^{\vee_n S^1}(\L_*)$.
Similarly, the outer automorphism group $\Out(F_n)$ acts on $\vee_n S^1$ up to homotopy and hence we obtain a representation of $\Out(F_n)$ on $HH^{\vee_n S^1}(\L)$.
While this result should at least morally be known to experts, the representations of $\Out(F_n)$ arising in this manner seem to have received little attention in the literature.
We will study a few special cases. The representations that we obtain inherit an additional filtration (the Hodge or Poincar\'e-Birkhoff-Witt filtration) such that the associated graded representation factors through $\GL(n,\Z)$.
We show that in general the representations of $\Out(F_n)$ thus obtained do not factor through $\GL(n,\Z)$, but are nontrivial iterated extensions of $\GL(n,\Z)$ representations.

In particular, it is an open problem to determine the lowest dimensional representations of $\Out(F_n)$  that do not factor through $\GL(n,\Z)$.\footnote{One assumes $n\geq 3$ as $\Out(F_2)=\GL(2,\Z)$.}
A lower bound has been obtained by D. Kielak \cite{Kielak}, who showed that the dimension must be at least
\[
 {n+1 \choose 2}.
\]
For $n=3$ the lower bound was refined to $7$ (instead of $6$) \cite{Kielak2}.
We obtain an upper bound as follows.

\begin{thm}\label{thm:outfrrep}
For $n\geq 3$, the representations of $\Out(F_n)$ on $HH^{\vee_n S^1}(\Chev(\alg g))$, where $\Chev(\alg g)$ is the Chevalley complex of a free Lie algebra $\alg g=\FreeLie(x)$ in one generator $x$ of odd degree, contain a direct summand representation which does not factor through $\GL(n,\Z)$ and has dimension
\[
 \frac{n(n^2+5)}{6}.
\]
In particular, for $n=3$ this representation saturates the lower bound $7$ obtained in~\cite{Kielak2}.
\end{thm}

 The previously known representations with such property have the smallest dimension 21 for $n=3$ and 
 \[
 (2^n-1){{n-1}\choose 2}
 \]
 for $n\geq 4$, see~\cite[Section~4]{Kielak}, and also~\cite{BV,GL}.

The higher Hochschild homology on spheres was introduced and studied in the original work of Pirashvili~\cite{pirash00} and on wedges of spheres it was studied in~\cite{Song,SongTur} in connection with the homology and homotopy of spaces of higher dimensional string links. An interesting feature of this homology is that it admits a decomposition into a direct product, and the factors of this {\it Hodge splitting} depend only on the parity of the dimensions of the spheres. In particular, if we know $HH^{\vee_n S^1}(\L_*)$ with the Hodge decomposition, we can reconstruct $HH^{\vee_n S^d}(\L_*)$ for any other odd~$d$. On the other hand, the 
homotopy type of a map $\vee_n S^d\to\vee_n S^d$, $d\geq 2$, is completely determined by the map in  homology. Therefore, $HH^{\vee_n S^d}(\L_*)$, $d\geq 2$, is acted upon by  the monoid $\End(\Z^n)$ of endomorphisms of $\Z^n$. For $d=1$, we get an action of the monoid  $\End(F_n)$ of endomorphisms of a free group $F_n$. In Section~\ref{s:CH_vee_n} we define a certain explicit complex $CH^{\vee_n S^1}(\L_*)$ computing $HH^{\vee_n S^1}(\L_*)$.

\begin{thm}\label{thm:endf_n_HP}
For any $\Gamma$-module $\L_*$, the action of $\End(F_n)$  on $HH^{\vee_n S^1}(\L_*)$ is naturally lifted on the level 
of the complex $CH^{\vee_n S^1}(\L_*)$. Moreover this action respects the Hodge splitting as an increasing filtration, and the action on the associated graded complex $\gr\, CH^{\vee_n S^1}(\L_*)$ factors through
$\End(\Z^n)$.
\end{thm}

 We will see in Section~\ref{s4} that as an $\End(\Z^n)$ module, $\gr\, HH^{\vee_n S^1}(\L_*)$
is (up to regrading) naturally isomorphic to $HH^{\vee_n S^d}(\L_*)$ for any odd $d\geq 3$.

The fact that the $\End(F_n)$ action above respects the Hodge filtration is actually a manifestation of a more general phenomenon. We show in Section~\ref{s4} that the Hodge filtration in $HH^{X_*}(\L_*)$,
that can also be called Poincar\'e-Birkhoff-Witt filtration,
  is defined functorially in $X_*$ and $\L_*$.  This filtration is an interesting phenomenon in itself that does not seem to appear earlier in any kind of functor 
calculus. 
  In particular, the Hodge filtration should not be confused with the  cardinality or rank (co)filtration considered, for example, in~\cite{AyFr,IntJMc}, 
and inspired from the manifold functor calculus~\cite{WeissEmb}, see Subsection~\ref{ss42+}. 
In that subsection we also explain in which sense the Hodge filtration in the Hochschild-Pirashvili homology
 on suspensions is exhaustive: it is dense in the topology induced by the cardinality cofiltration. 
   
  Theorem~\ref{thm:endf_n_HP} can be \lq\lq{}categorified\rq\rq{} to all suspensions and maps between them.
More specifically,
let $\TOP_*$ 
denote the category of pointed topological spaces with morphisms homotopy classes of pointed maps. 
Let $\TOP_*|_\Sigma$ denote its full subcategory whose objects are suspensions.  By $\Sigma(\TOP_*)$ we denote the image category of the suspension functor  $\Sigma\colon\TOP_*\to\TOP_*$.
Notice  that  any suspension is rationally equivalent to a wedge of 
spheres~\cite[Theorem~24.5]{FHT1}. Thus, for the sake of concreteness and slightly simplifying the matters, the reader can think about the category $\TOP_*|_\Sigma$ as about the full subcategory in $\TOP_*$  of wedges of spheres of possibly different dimensions $\geq 1$. 
The following theorems generalize Theorem~\ref{thm:endf_n_HP} on this category $\TOP_*|_\Sigma$.

\begin{thm}\label{t:HP_suspensions1}
For any right $\Gamma$-module $\L_*$, the cofunctor $HH^{(-)}(\L_*)\colon {\TOP_*}^{op}
\to gVect$ admits an increasing  filtration generalizing the Hodge filtration on $HH^{\vee_n S^1}(\L_*)$,
such that the completed associated graded  functor $\gr\, HH^{(-)}(\L_*)$ restricted on $\TOP_*|_\Sigma$ factors 
through the reduced homology functor $\tilde H_*\colon \TOP_*\to gVect$.  Over $\Sigma(\TOP_*)$,
this filtration splits in the sense that one has a natural isomorphism 
$HH^{(-)}(\L_*)|_{\Sigma(\TOP_*)} \to \gr\, HH^{(-)}(\L_*)|_{\Sigma(\TOP_*)}$.
\end{thm}

In Section~\ref{ss43} we construct a cofunctor $\CH^{(-)}(\L_*)\colon (\TOP_*|_\Sigma)^{op}
\to dgVect$.

\begin{thm}\label{t:HP_suspensions2}
The cofunctor  $\CH^{(-)}(\L_*)\colon (\TOP_*|_\Sigma)^{op}
\to dgVect$ has the following properties
\begin{itemize}
\item $H_*\circ \CH^{(-)}(\L_*) = HH^{(-)}(\L_*)$. 
\item The complex $\CH^{\vee_n S^1}(\L_*)$ is identical to  $CH^{\vee_n S^1}(\L_*)$.
\item This functor admits an increasing  (Hodge) filtration compatible with the Hodge filtration in homology.
\item The completed associated graded functor $\gr\,\CH^{(-)}(\L_*)$ factors through the reduced homology finctor $\tilde H_*\colon \TOP_*|_\Sigma\to gVect$.
\item Over $\Sigma(\TOP_*)$, the Hodge filtration in $\CH^{(-)}(\L_*) $ splits in the sense that one has
a natural isomorphism $\CH^{(-)}(\L_*)|_{\Sigma(\TOP_*)} \to \gr\, \CH^{(-)}(\L_*)|_{\Sigma(\TOP_*)}$.
\end{itemize}
\end{thm}

More concretely when we say that the functors $\gr\, HH^{(-)}(\L_*)\colon \TOP_*|_\Sigma \to gVect$ and
$\gr\, \CH^{(-)}(\L_*)\colon \TOP_*|_\Sigma\to dgVect$ factor through $\tilde H_*\colon \TOP_*|_\Sigma\to gVect$ we mean that for any pointed space $Y_*$,  both $\gr\, HH^{\Sigma Y_*}(\L_*)$ and $\gr\, \CH^{\Sigma Y_*}(\L_*)$ can be described as a power series expression in $\tilde H_*\Sigma Y_*$:
\begin{gather}
\gr\, HH^{\Sigma Y_*}(\L_*)=
\prod_n\Hom_{\bbS_n}\left( (\tilde H_*\Sigma Y_*)^{\otimes n}, {\mathcal H}_{\L_*}(n) \right),
\label{eq:power_ser1}  \\
\gr\, \CH^{\Sigma Y_*}(\L_*)= 
\prod_n\Hom_{\bbS_n}\left( (\tilde H_*\Sigma Y_*)^{\otimes n}, {\mathcal C}_{\L_*}(n) \right),
\label{eq:power_ser2}
\end{gather}
where $\mathcal{C}_{\L_*}$ is some symmetric sequence in chain complexes depending on $\L_*$, and 
$\mathcal{H}_{\L_*}$ is its homology symmetric sequence. The fact that the Hodge filtration splits over 
$\Sigma(\TOP_*)$ means that we have isomorphisms
\begin{gather}
\CH^{\Sigma Y_*}(\L_*)\stackrel{\simeq}{\longrightarrow} \gr\, \CH^{\Sigma Y_*}(\L_*),\label{eq:spl1}\\
HH^{\Sigma Y_*}(\L_*)\stackrel{\simeq}{\longrightarrow} \gr\, HH^{\Sigma Y_*}(\L_*)\label{eq:spl2}
\end{gather}
natural in $\Sigma Y_*\in \Sigma(\TOP_*)$. 
The $n$-th term of the Hodge splitting is exactly the $n$-th factor in~\eqref{eq:power_ser1} and~\eqref{eq:power_ser2}. (This splitting also means that the higher Hochschild complexes for suspensions split as a product of complexes.) In case a pointed map $f\colon \Sigma Y_*\to \Sigma Z_*$ is not a suspension, the Hodge splitting in the higher Hochschild complexes/homology (via isomorphisms~\eqref{eq:spl1}-\eqref{eq:spl2})  behaves like a filtration: higher  terms of the splitting can be send  non-trivially to lower  ones.   
 In Section~\ref{s:rht_map} we compute how from the given rational homotopy type of a map of suspensions
one gets the induced map between the terms of the splitting. We also demonstrate  this  on some examples, such as
   the Hopf map $S^3\to S^2$ and  a
non-trivial pointed map $S^2\to S^2\vee S^1$. 

Some of the techniques that we develop for suspensions work equally well for general spaces.
In Section~\ref{s:non_susp} we briefly consider this general case of non-suspensions.
Theorems~\ref{th:non_susp}-\ref{th:non_susp2} and Proposition~\ref{p:non_susp} describe these more general higher Hochschild complexes
in the case $\L_* = M\otimes C^{\otimes \bullet}$ as some kind of homotopy base change type of
Chevalley complexes. In this section we also show that for a connected pointed space $X_*$ (of finite type) the Hodge filtration splits for any coefficient $\Gamma$ module $\L_*$ if and only if $X_*$ is rationally homology equivalent to a suspension.

\subsection*{Acknowledgements}
We thank G. Arone, B. Fresse, G. Ginot, and D. Kielak for helpful discussions. V.T. thanks the MPIM and the IHES, where he spent his sabbatical and where he started to work on this project.  T.W. has been partially supported by the Swiss National Science foundation, grant 200021\_150012, and the SwissMAP NCCR funded by the Swiss National Science foundation.

\section{Special case of $\End(F_n)$ action}\label{s2}
In this section we look at the special case $\L_*=M\otimes C^{\otimes\bullet}$, where $C$ is a cocommutative coalgebra and $M$ a $C$-comodule as before. If not otherwise stated we will always assume that $C$ is simply connected. We will define a
complex $CH^{\vee_n S^1}(M\otimes C^{\otimes \bullet})$ and an $\End(F_n)$ action on it. 
In Section~\ref{s:CH_vee_n} we explain why this complex computes $HH^{\vee_n S^1}(M\otimes C^{\otimes \bullet})=HH^{\vee_n S^1}(C,M)$ and why the $\End(F_n)$ action that we define corresponds to the topological action. Define $CH^{\vee_n S^1}(M\otimes C^{\otimes \bullet})$ as 
$M\otimes (\Omega C)^{\otimes n}$, where $\Omega C$ is the cobar construction of $C$ --- as a space it is a free associative algebra generated by $C[1]$. The differential
\begin{equation}\label{eq:differential1}
d=d_M+d_C+\delta,
\end{equation}
where $d_M$ and $d_C$ are induced by the differentials on $M$ and $\Omega C$ respectively and
\[
\delta(m\otimes b_1\otimes\ldots\otimes b_n)=
\sum\sum_j\pm m'\otimes\ldots\otimes [m'',b_j]\otimes\ldots\otimes b_n,
\]
where we used Sweedler\rq{}s notation; $\pm$ is the Koszul sign due to permutation of $m''$ with $b_i$\rq{}s. 

We can assume without loss of generality that $C=\Chev(g)$ is the Chevalley complex of a dg Lie algebra $\alg g$ concentrated in strictly positive degrees. (If not, take for $\alg g$ the Harrison complex of $C$.) As a cocommutative coalgebra it is freely cogenerated by ${\alg g} [-1]$. 
In the latter case the aforementioned complex is quasi-isomorphic to 
$ M\otimes (\mU \alg g)^{\otimes n}$ , where $\mU \alg g$ is the universal envelopping algebra of $\alg g$, with differential
\begin{equation}\label{eq:differential2}
d=d_M+d_{\alg g}+\delta,
\end{equation}
defined similarly:
 $d_M$ and $d_{\alg g}$ are induced from the differentials on $M$ and $\alg g$ and 
\[
\delta (m\otimes b_1,\dots ,b_n) = \sum \sum_j \pm m'\otimes b_1\otimes\ldots \otimes [\pi(m''), b_j]\otimes \dots \otimes b_n,
\]
where  $\pi:\Chev(\alg g)\to \alg g$ is the projection to the cogenerators.

The action of $\End(F_n)$ on $ M\otimes (\mU \alg g)^{\otimes n}$ and $M\otimes (\Omega C)^{\otimes n}$
is described by the same formulas. Both $\mU \alg g$ and $\Omega C$ are cocommutative Hopf algebras. 
In Sweedler\rq{}s notation the iterated coproduct is written as
\[
\Delta^k b=\sum b^{(1)}\otimes b^{(2)}\otimes\ldots\otimes b^{(k)}.
\]
Since the coproduct is cocommutative, we will be writing instead
\[
\Delta^k b=\sum b^{(\bullet)}\otimes b^{(\bullet)}\otimes\ldots\otimes b^{(\bullet)}.
\]
Let $\Psi\in\End(F_n)$ send
\begin{equation}\label{eq_Psi1}
x_i\mapsto x_{\alpha_{i1}}^{\varepsilon_{i1}}\cdot x_{\alpha_{i2}}^{\varepsilon_{i2}}\cdot 
\ldots\cdot x_{\alpha_{ik_i}}^{\varepsilon_{ik_i}},\quad i=1\ldots n,
\end{equation}
where $\varepsilon_{ij}=\pm 1$, $\alpha_{ij}\in\{1\ldots n\}$. We let $\beta_{ij}=\frac{1-\varepsilon_{ij}}2\in\{0,1\}$ and  define
\begin{equation}\label{eq_Psi2}
\Psi^*(m\otimes b_1\otimes\ldots\otimes b_n):=
m\otimes\sum\pm\bigotimes_{i=1}^n\prod_{j=1}^{k_i}s^{\beta_{ij}}(b_{\alpha_{ij}}^{(\bullet)}),
\end{equation}
where the sign $\pm$ is the Koszul sign arising from the factors permutations, $s$ is the antipod. 

\begin{ex}\label{ex_Psi}
(a) $n=1$; $x_1\mapsto (x_1)^2$. 
\[
\Psi^*(m\otimes b)=m\otimes \sum b'\cdot b''.
\]
(b) $n=1$; $x_1\mapsto x_1^{-1}$.
\[
\Psi^*(m\otimes b)=m\otimes s(b).
\]
(c) $n=2$; $x_1\mapsto x_1\cdot x_2$, $ x_2\mapsto x_2$.
\[
\Psi^*(m\otimes b_1\otimes b_2)= m\otimes \sum b_1\cdot b_2'\otimes b_2''.
\]
\end{ex}

\begin{prop}\label{p:out_action1}
The formula \eqref{eq_Psi2} defines the right action of $\End(F_n)$ on the complexes 
$ M\otimes (\mU \alg g)^{\otimes n}$ and $M\otimes (\Omega C)^{\otimes n}$.
\end{prop}

\begin{proof}
To see that $\Psi^*$ is a morphism of complexes we notice that it commutes with each term of the 
differentials~\eqref{eq:differential1} and \eqref{eq:differential2}: it commutes with $d_M$ by obvious reasons; 
it commutes with $d_{\alg g}$ since both product and coproduct of $\mU \alg g$ are morphisms of complexes; it commutes with $\delta$ since both product and coproduct respect the $\alg g$ action. 

For the composition, it is quite easy to see that $(\Psi_1\circ \Psi_2)^*=\Psi_2^*\circ\Psi_1^*$, where the composition $\Psi_1\circ\Psi_2$ is understood as substitution without simplification. We only need to check 
that in case  $(\Psi_1\circ\Psi_2)(x_i)$ has two consecutive factors $x_j$ and $x_j^{-1}$ for some $i$,
then $(\Psi_1\circ \Psi_2)^*$ is the same as if these factors are cancelled out. But in such case,
$(\Psi_1\circ \Psi_2)^*(m\otimes b_1\otimes\ldots\otimes b_n)$ also has two consecutive factors
$b_j^{(\bullet)}$ and $s(b_j^{(\bullet)})$, which can also be eliminated:
\[
\sum b_j^{(\bullet)}\cdot s(b_j^{(\bullet)})\otimes (b_j^{(\bullet)})^{\otimes k}=
1\otimes \sum (b_j^{(\bullet)})^{\otimes k}=1\otimes\Delta^k b_j=
\sum  s(b_j^{(\bullet)})\cdot b_j^{(\bullet)} \otimes (b_j^{(\bullet)})^{\otimes k}.
\]

\end{proof}

\subsection{Hodge decomposition/filtration}\label{ss:hodge_filtr}

The Poincar\'e-Birkhoff-Witt isomorphism $S\alg g\to\mU \alg g$ respects both the coalgebra and $\alg g$ action structures. As a corollary, the induced map
\[
 M\otimes (S \alg g)^{\otimes n}\to  M\otimes (\mU \alg g)^{\otimes n}
 \]
 is an isomorphism of complexes. The image of the subcomplex $M\otimes S^{m_1}\alg g\otimes\ldots
 \otimes S^{m_n}\alg g$ in $M\otimes (\mU \alg g)^{\otimes n}$ is called $(m_1,\ldots,m_n)$ Hodge 
 multidegree component, whose {\it total Hodge degree} is $m=m_1+\ldots +m_n$. 
 One has
 \[
 \bigoplus_{m_1+\ldots+m_n=m}M\otimes S^{m_1}\alg g\otimes\ldots
 \otimes S^{m_n}\alg g=M\otimes S^m(H^1\otimes\alg g),
 \]
 where $H^1:=H^1(\vee_n S^1,\Z)=\Z^n$ viewed as a  space concentrated in degree zero.
 Below $H_1:=H_1(\vee_n S^1,\Z)$.
 
 \begin{prop}\label{p:out_hodge_filtr}
 The action of $\End(F_n)$ on $M\otimes (\mU \alg g)^{\otimes n}$  preserves the total Hodge degree
 as a filtration. The induced action on the associated graded complex $\gr\, M\otimes (\mU \alg g)^{\otimes n}$ factors through $\End(H_1)=\End(\Z^n)$ as one has 
 \[
 \gr\, M\otimes (\mU \alg g)^{\otimes n}= M\otimes S(H^1\otimes \alg g).
 \]
 \end{prop}
 
 This proposition is a particular case of Theorem~\ref{thm:endf_n_HP}.
 
 \begin{proof}
 The  Hodge filtration is preserved because both the product and coproduct of $\mU \alg g$ preserve
 the Poincar\'e-Birkhoff-Witt filtration. Notice also that if we apply~\eqref{eq_Psi2} to define an $\End(F_n)$ action on $M\otimes (S \alg g)^{\otimes n}$, we get exactly $M\otimes (S \alg g)^{\otimes n}\simeq M\otimes S(H^1\otimes \alg g)$ as a 
 right $\End(F_n)$ module.
 \end{proof}

\begin{rem}\label{r:HH_suspension}
It will be shown in Subsection~\ref{ss41} (see Remark~\ref{r:grCH_C_M}) that for any pointed space $Y_*$ of finite type, the Hochschild-Pirashvili homology
$HH^{\Sigma Y_*}(C(\alg g),M)$ is computed by the complex
\[
M\otimes S(\tilde H^*Y_*\otimes\alg g),
\]
where $\tilde H^*(Y)$ is the reduced cohomology of $Y$ viewed as a negatively graded vector space.
The differential has the same form~\eqref{eq:differential2}.\footnote{Unless certain convergency 
properties are satisfied, $S(-)$ should be undersood as a completed symmetric algebra, i.e. a direct product
$\prod_{m\geq 0} S^m(-)$ rather than a direct sum. Similarly the tensor product should be understood as the completed tensor product with respect to the homological degree of $\alg g$.}
\end{rem}

\section{$\Out(F_n)$ representations. Proof of Theorem~\ref{thm:outfrrep}}\label{s:out_n_rep}
Recall isomorphism \eqref{eq:base}, which in particular implies that in case $M=C$  the  action of $\Aut(F_n)$ on $HH^{\vee_n S^1}(C,M)=HH^{\vee_n S^1}(C)$ descends to an $\Out(F_n)$ action. 
%
%
Recall also that according to Proposition~\ref{p:out_hodge_filtr} the higher Hochschild homology $HH^{\vee_n S^1}(C, M)$ carries a Hodge filtration such that the action of $\Aut(F_n)$ on the associated graded factors through $\GL(n,\Z)$.
In other words, all $\Aut(F_n)$ and $\Out(F_n)$ modules obtained in this manner can be obtained by iterated extension of $\GL(n,\Z)$-modules by $\GL(n,\Z)$-modules.

\subsection{Example 1: Polynomial coalgebras}
If $C=\Q[x_1,\dots,x_n]$ is a cofree cocommutative coalgebra (in potentially odd generators), we have $\alg g=\xi_1\Q \oplus \cdots \oplus \xi_n\Q$ as abelian Lie algebra, where the generators $\xi_j$ are degree shifted by one unit with respect to the generators $x_j$.
In this case the Hodge grading is preserved by the $\Aut(F_n)$ action (because $\mU \alg g$ is commutative) and hence all representations obtained factor through $\GL(n,\Z)$.
Since the differential on $C\otimes(\mU \alg g)^{\otimes n}$ vanishes the higher Hochschild homology is just
\[
HH^{\vee_n S^1}(C)\cong C\otimes S(  H^1 \otimes \alg g)
\]
with the $\Out(F_n)$ action factoring through $GL(n,\Z)=GL(H_1)$, which acts by the standard action on $\Z^n=H^1$.

\subsection{Example 2: Dual numbers}\label{sec:exdualnumbers}
Consider the coalgebra of dual numbers $\Q\oplus x\Q$, where $x$ is a primitive cogenerator of even degree. The (Koszul) dual Lie algebra is the free Lie algebra in one odd generator $\xi$, i.e., $\alg g =\xi \Q \oplus [\xi,\xi] \Q$.
Then the associated graded of $CH^{\vee_n S^1}(C\otimes C^{\otimes\bullet})$ may be identified with 
\[
\gr\, CH^{\vee_n S^1}(C\otimes C^{\otimes\bullet})\cong C\otimes S(H^1\otimes \alg g)\cong C \otimes \Q[\xi_1,\dots, \xi_n, \eta_1,\dots, \eta_n].
\]
Here $\xi_j$ corresponds to $\xi$ on the $j$-th circle and $\eta_j$ corresponds to $\eta=[\xi,\xi]=2\xi^2$ on the $j$-th circle. Notice that $ad_\xi(\xi)=\eta$ and $ad_\xi(\eta)=0$.
The complex has length~2:
\[
0\leftarrow 1\otimes \Q[\xi_1,\dots, \xi_n, \eta_1,\dots, \eta_n]{\stackrel d\longleftarrow} x\otimes \Q[\xi_1,\dots, \xi_n, \eta_1,\dots, \eta_n]\leftarrow 0.
\]
The differential is defined such that
\begin{multline*}
d(x\otimes P(\xi_1,\dots, \xi_n, \eta_1,\dots, \eta_n) )= \sum_{j=1}^n 1\otimes  ad_{\xi_j} P(\xi_1,\dots, \xi_n, \eta_1,\dots, \eta_n)=\\
= \sum_{j=1}^n 1\otimes  \eta_j \frac{\partial}{\partial \xi_j} P(\xi_1,\dots, \xi_n, \eta_1,\dots, \eta_n)
\end{multline*}
The differential can be identified with the de Rham differential on an $n$-dimensional odd vector space, identifying $\eta_j$ with $d_{dR}\xi_j$.
One can identify the corresponding representations of $\GL(n,\Z)$. 
Namely, if we fix in the associated graded the Hodge degree to be $m$, then the corresponding representations of $\GL(n,\Z)$ one obtains correspond to partitions of the form $m=\ell + 1+\cdots +1$.
To be precise the homology is the sum $U^I\oplus U^{II}$, where $U^I=\coker\, d$, $U^{II}=\ker d$. The part of degree $k$ in $\xi$ and $\ell$ in $\eta$ is sent by $d$ to the part of degree $k-1$ 
in $\xi$ and $\ell+1$ in $\eta$:
\[
0\leftarrow \Lambda^{k-1} H^1\otimes S^{\ell+1} H^1{\stackrel d\longleftarrow}\Lambda^{k}H^1\otimes S^{\ell}H^1\leftarrow 0.
\]
The $\GL(n)$ module $\Lambda^k H^1\otimes S^{\ell} H^1$ is a direct sum of 2 representations encoded by partitions $(\ell+k)=\ell+1+\ldots+1$  and $(\ell+k)=(\ell+1)+1+\ldots+1$. We conclude that the kernel of $d$ in this bigrading is $V_{(\ell,1^k)}$ and the cokernel of $d$ is $V_{(\ell+2,1^{k-2})}$. The bigrading by $\xi$ and $\eta$ is preserved in $C\otimes (\U\alg g)^{\otimes n}$ only as a filtration. Instead one can consider the {total $\xi$ grading} by assigning 1 to each $\xi$ and 2 to each $\eta=[\xi,\xi]$. 
The component $U^I_N\oplus U^{II}_N$ in the homology of total $\xi$ degree $N$ is a filtered space, whose associated graded is
\[
\gr U^I_N=\bigoplus_{2\ell+k=N+1}V_{(\ell,1^k)},\quad
\gr  U^{II}_N=\bigoplus_{2\ell+k=N}V_{(\ell,1^k)}.
\]
For both $U^I$ and $U^{II}$ the Hodge degree of $V_{(\ell,1^k)}$ is $\ell+k$.

\subsection{The lowest non-trivial example worked out}\label{sec:lowest}
Let us consider the first $\Out(F_n)$ representation obtained by the above methods that does not factor through $\GL(n,\Z)$.
It is obtained as the cokernel of the differential in the dual numbers example above for $n=3$ and the total $\xi$ degree~3. It was denoted by $U^I_3$ in the previous subsection.
The representation is 7 dimensional.
As in Subsection \ref{sec:exdualnumbers} one sees that the associated graded representation splits into two $\GL(3,\Z)$ representations
\[
\gr U_{3}^I = V_{(2)} \oplus V_{(1,1,1)}.
\]
In other words, $U_{3}^I$ is an extension
\[
0\to  V_{(2)} \to U_{3}^I \to V_{(1,1,1)} \to 0.
\]
A representative of the cohomology class in $HC^{\vee_3 S^1}(C)$ spanning the $V_{(1,1,1)}$ part is 
\[
e := 1\otimes \xi\otimes \xi \otimes \xi.
\]
Representatives forming a basis of $V_{(2)}$ are
\begin{align*}
f_1&:= 1\otimes [\xi,\xi] \otimes \xi\otimes 1 \cong  -1\otimes \xi \otimes [\xi,\xi] \otimes  1 &
f_2&:= 1\otimes [\xi,\xi] \otimes 1\otimes \xi  \\
f_3&:= 1\otimes 1 \otimes [\xi,\xi] \otimes \xi &
f_4&:= 1\otimes [\xi,\xi] \xi \otimes 1 \otimes 1 \\
f_5&:= 1\otimes 1\otimes [\xi,\xi] \xi \otimes 1  &
f_6&:= 1\otimes 1\otimes 1\otimes [\xi,\xi] \xi.
\end{align*}

\subsection{The proof of Theorem \ref{thm:outfrrep}}
More generally let us consider representation $U^I_{3}$ of $\Out(F_n)$  for arbitrary $n\geq 3$.
We claim that this representation satisfies the requirements of Theorem \ref{thm:outfrrep}, i.e., it does not factor through $\GL(n,\Z)$ and it has dimension $\frac{n(n^2+5)}{6}$.

Indeed, as in Subsection~\ref{sec:lowest} we can identify the associated graded representation under the Hodge filtration with 
\[
 \gr U^I_{3} = V_{(2)}\oplus V_{(1,1,1)}
\]
where $V_{(2)}$ and $V_{(1,1,1)}$ are the irreducible representations of the linear group $\GL(n)$ corresponding to the partitions $(2)$ and $(1+1+1)$. Hence we find that indeed
\[
 \dimens U^I_{3} 
 = \dimens V_{(2)} + \dimens V_{(1,1,1)}
 = \frac{n(n+1)}{2}+ {n \choose 3} = \frac{n(n^2+5)}{6}.
\]

Next, we check that the representation does not factor through $\GL(n)$.  Consider $E_{12}, E_{\bar 1 \bar 2}\in \Out(F_n)$ that send 
\[
E_{12}(x_i)=
\begin{cases}
x_1x_2,& i=1;\\
x_i,& \textrm{otherwise;}
\end{cases}
\qquad
E_{\bar 1 \bar 2}(x_i)=
\begin{cases}
x_2x_1,& i=1;\\
x_i,& \textrm{otherwise.}
\end{cases}
\]
We will show that the action of $E_{12}$ is different from that of $E_{\bar 1 \bar 2}$  in the representation $U^I_{3}$ for $n\geq 3$.
Indeed, choosing basis vectors as in Subsection \ref{sec:lowest} we find that 
\[
 E_{12} \cdot (1, \xi, \xi, \xi,1,\dots,1) = (1, \xi, \xi, \xi,1,\dots,1)+ {\frac 12} (1, [\xi, \xi],1, \xi,1,\dots,1)
\]
while 
\[
 E_{\bar 1 \bar 2} \cdot (1, \xi, \xi, \xi,1,\dots,1) = (1, \xi, \xi, \xi,1,\dots,1)- {\frac 12}(1, [\xi, \xi],1, \xi,1,\dots,1).
\]
To recall $U^I_{3}$ is the cokernel of $d$. Thus 
we  need to verify that $(1, [\xi, \xi],1, \xi,1,\dots,1)\in \Q\otimes (S\alg g)^{\otimes n}$ is not in the image of $d$. As we have seen in Subsection~\ref{sec:exdualnumbers}, $d$ is the de Rham differential which is acyclic on non-constant polynomials in $\xi_i$ and $\eta_j$, thus we only have to check that the corresponding polynomial is not de Rham closed: 
\[
\sum_{j=1}^n   \eta_j \frac{\partial}{\partial \xi_j} (\eta_1\xi_3) =\eta_1\eta_3\neq 0.
\]
\hfill\qed

\subsection{Bead representations}
Generalizing the example of dual numbers we may consider the coalgebra
\[
C_N=\Q\oplus x_1\Q\oplus x_2\Q\oplus\ldots\oplus x_N\Q,
\]
where the cogenerators $x_i$ are of even degrees and primitive.
The Koszul dual Lie algebra is again free
\[
\alg g=FreeLie(\xi_1,\dots, \xi_N).
\]
There is a $\Z^N$ grading on $C_N$ and a representation of $\bbS_N$, and hence a similar grading and action on the higher Hochschild homology $HH^{\vee_n S^1}(C_N)$.
We may introduce a representation of $\Out(F_n)$ for every irreducible representation $V_\lambda$ of $\bbS_N$ labelled by a partition $\lambda$ of $N$: 
\[
U_\lambda = HH^{\vee_n S^1}(C_N)^{1,\dots, 1} \otimes_{\bbS_N} V_\lambda
\]
Here the superscript $(\cdot)^{1,\dots, 1}$ shall mean that we pick out the piece of $\Z^N$-degree $(1,\dots, 1)$.
We will call $U_\lambda$ the \emph{bead representation}\footnote{The name stems from the fact that elements of $\Omega C_N$ can be understood as linear combinations of configurations of beads of $N$ colors arranged on a string.} of $\Out(F_n)$ associated to the partition $\lambda$. Notice that the obtained complex is again of length~2.
Thus we have again $U_\lambda=U_\lambda^I\oplus U_\lambda^{II}$ where $U_\lambda^I$ is the cokernel of the differential and $U_\lambda^{II}$ is the kernel. We will call $U_\lambda^I$ the bead representation of first type and $U_\lambda^{II}$ the bead representation of second type. \footnote{The representations $U_N^{I,II}$ considered in Subsection~\ref{sec:exdualnumbers} correspond to
$U_{(N)}^{I,II}$ in the new notation.}

 {\bf Open problem:} Describe $U_\lambda$. In particular, what are the dimensions $\mathit{dim}(U_\lambda^{I,II})$? If we decompose the associated graded $\gr U_\lambda$ into irreducible representations of $\GL(n,\Z)$ (actually $\GL(n,\R)$) 
 \[
 \gr U_\lambda \cong \oplus_\mu V_\mu
 \]
 which partitions $\mu$ occur in the direct sum, with what multiplicity?

\section{Complexes $CH^{\vee_n S^1}(\L_*)$. Proof of Theorem~\ref{thm:endf_n_HP}}\label{s:CH_vee_n}

Recall that
in case the space $X$ (respectively pointed space $X_*$) is obtained as a realization of a (pointed) finite simplicial set
$\X_\bullet\colon\Delta^{op}\to \Fin$ (respectively $\X_\bullet\colon \Delta^{op}\to\Gamma$), the higher Hochschild homology $HH^X(\L)$ (respectively $HH^{X_*}(\L_*)$) can be computed as the homology of the totalization of the cosimplicial chain complex $\L\circ \X\colon\Delta\to dgVect$ (respectively 
 $\L_*\circ \X_*\colon\Delta\to dgVect$).   
  The same construction works for realizations of bisimplicial (and more generally  multisimplicial) sets. Indeed, if $\X_{\bullet\bullet}$ is a bisimplicial set, then its realization $|\X_{\bullet\bullet}|$ is homeomorphic to the realization 
 $|\diag(\X_{\bullet\bullet})|$ of its diagonal simplicial set. On the other hand, one also has the 
 Eilenberg-Zilber quasi-isomorphism
 \begin{equation}\label{eq_EZ}
 \Tot(\diag \L\circ \X_{\bullet\bullet}){\stackrel {EZ}\longrightarrow}\Tot(\L\circ\X_{\bullet\bullet}).
 \end{equation}
 As the first complex computes the Hochschild-Pirashvili homology of $|\diag(\X_{\bullet\bullet})|=
 |\X_{\bullet\bullet}|$, so does the second. 
 
 Now notice that the complexes $M\otimes\left(\Omega C\right)^{\otimes n}$ can be obtained as totalization
 of an $n$-multicosimplicial chain complex (rather than just cosimplicial). (In fact its diagonal totalization is $M\otimes\Omega\left(C^{\otimes n}\right)$.) The corresponding multicosimplicial complex is
 obtained as the composition of $M\otimes C^{\otimes\bullet}$ with an $n$-multisimplicial model 
 of $\vee_n S^1$. Let $S_\bullet^1$ denote the standard simplicial model for $S^1$: its set of $k$-simplices
 consists of a basepoint $*$ and also all monotonic non-constant sequences of $0$\rq{}s and $1$\rq{}s of length $k+1$. 
 This set can be identified with $k_*$ (where $i\in k_*$ corresponds to a sequence with $i$ $1$\rq{}s). The $n$-multisimplicial model for $\vee_n S^1$, 
 we denote it by $(\vee_n S^1)_{\underbrace{\bullet\ldots\bullet}_n}$,
 is obtained as a degreewise wedge of $n$ $n$-multisimplical sets. The $i$-th summand of the wedge is the product of $S_\bullet^1$ and $(n-1)$ constant one-point simplicial sets, with $S_\bullet^1$ appearing on the $i$-th place in the product. Notice that the $(k_1,k_2,\ldots,k_n)$ component of 
 $(\vee_n S^1)_{\underbrace{\bullet\ldots\bullet}_n}$ is the set $\bigvee_{i=1}^n (k_i)_*\simeq
 (k_1+\ldots +k_n)_*$.  Thus the totalization of our multicosimplial complex is
 \begin{equation}\label{eq_CH_vee_n}
CH^{\vee_n S^1}(\L_*):=\Tot(\L_*\circ (\vee_n S^1)_{\underbrace{\bullet\ldots\bullet}_n})
=\left(\prod_{(k_1,\ldots,k_n)} N\L_*\left(\Sigma_{i=1}^n k_i\right)[\Sigma_{i=1}^n k_i],\, d=d_1+\ldots+d_n\right),
\end{equation}
where
\begin{equation}\label{eq_NL}
NL_*(k)=\bigcap_{i=1}^k \ker s_i^*,
\end{equation}
and $s_i^*\colon \L(k_*)\to\L(k_*\setminus\{i\})$ is the map induced by the inclusion 
$$s_i\colon
k_*\setminus\{i\}\subset k_*.$$

The action of $\End(F_n)$ on $CH^{\vee_n S^1}(\L_*)$ is defined analogously as that on $CH^{\vee_n S^1}(M\otimes C^{\otimes\bullet})=M\otimes(\Omega C)^{\otimes n}$, see~\eqref{eq_Psi2}.\footnote{Recall that we assume that $C$ is simply connected. If we only assume that $C$ is connected, than the complex 
$CH^{\vee_n S^1}(M\otimes C^{\otimes\bullet})$ is $M\hat \otimes(\Omega C)^{\otimes n}$, where instead of the cobar complex we take the completed cobar and instead of tensor product the completed tensor product.} Notice that the coproduct on $\Omega C$ is the sum of coshuffles, and the product is just concatenation. 
 Let $\gamma$ lie in the $(k_1,\ldots,k_n)$
component of~\eqref{eq_CH_vee_n}, and $\Psi\in \End(F_n)$ is such that $x_j$ appears in total $r_j$ times
in $\Psi(x_1),\,\Psi(x_2),\ldots,\,\Psi (x_n)$. One has that $\Psi^*(\gamma)$ is the sum of $r_1^{k_1} 
\cdot r_2^{k_2}\cdot\ldots\cdot r_n^{k_n}$ elements each of which is obtained from $\gamma$ by
some permutation of its inputs. More concretely, $\Psi$ defines a map $\vee_n S^1\to \vee_n S^1$ such that any point on the $i$-th circle 
has exactly $r_i$ preimages. We put $k_1$ points on the first circle in the target wedge, $k_2$ on the second,
$\ldots$, $k_n$ on the last one. These points correspond to the inputs of $\gamma$. For
every point in the target we choose a preimage point (thus for the $i$-th circle there are $r_i^{k_i}$ choices
making the total of $\prod_{i=1}^n r_i^{k_i}$ choices). For every such choice we get a collection of points on the source wedge, which contributes a summand in $\Psi^*(\gamma)$, that has to be taken with the 
sign of permutation of inputs of $\gamma$. 

Consider  examples similar to those given in Example~\ref{ex_Psi}: 

(a) $n=1$; $\Psi(x_1)=x_1^2$. In this case,
\[
\Psi^*(\gamma(x_{11},\ldots,x_{1k_1}))=
\sum_{i=0}^{k_1}\sum_{\sigma\in \Sh(i,k_1-i)}(-1)^\sigma \gamma(\sigma(x_{11},\ldots,x_{1k_1})).
\]
Here  and below $\Sh(i,j)$ denotes the set of shuffles of an $i$-elements set with a $j$-elements set.

(b) $n=1$, $\Psi(x_1)=x_1^{-1}$. In this case 
\[
\Psi^*(\gamma(x_{11},\ldots,x_{1k_1}))=
(-1)^{\frac{k_1(k_1-1)}2}\gamma(x_{1k_1},\ldots,x_{11}).
\]

(c)  $n=2$; $\Psi(x_1)=x_1x_2$, $\Psi(x_2)=x_2$:
\[
\Psi^*(\gamma(x_{11},\ldots,x_{1k_1},x_{21},\ldots,x_{2k_2}))=
\sum_{i=0}^{k_2}\sum_{\sigma\in \Sh(i,k_2-i)}(-1)^\sigma \gamma(x_{11},\ldots,x_{1k_1},
\sigma(x_{1k_1+1},\ldots,x_{1k_1+i},x_{21},\ldots,x_{2k_2-i})).
\]

\begin{prop}\label{pr:act_topol}
The action of $\End(F_n)$ on $CH^{\vee_n S^1}(L_*)$ defined above coincides in the homology with the topological action.
\end{prop}

\begin{proof}[Idea of the proof]
One can check that for all elements $\Psi\in \End(F_n)$ their action $\Psi^*$ on $CH^{\vee_n S^1}(L_*)$
can be decomposed into a composition of  maps induced by multisimplicial maps,  Eilenberg-Zilber maps~\eqref{eq_EZ}, and some natural chain homotopy inverses to those maps. 
\end{proof}

This proposition is a partial case of Theorem~\ref{t:HP_suspensions2}. That\rq{}s why we choose not to give 
a detailed proof of it, but only mention that there is a proof which goes through a careful study of 
 multi-simplical maps. (This argument is similar to the explicit identification of the surface product studied in~\cite{GTZ0}.) Indeed, Theorem~\ref{t:HP_suspensions2} among other things states that
 the complexes $CH^{\vee_n S^1}(L_*)$ are identical to $\CH^{\vee_n S^1}(L_*)$, where the latter
 ones are constructed using the definition of the Hochschild-Pirashvili homology in terms of derived maps 
 of right $\Gamma$ modules. Moreover, Remark~\ref{rem:action_ident} asserts that the induced action of 
 $\End(F_n)$ on $\CH^{\vee_n S^1}(L_*)$ is identical to the one  on $CH^{\vee_n S^1}(L_*)$
 defined in this section.
  We will also see in Subsection~\ref{ss44} that the reason 
that the $\End(F_n)$ action on  $HH^{\vee_n S^1}(L_*)$ can be lifted on the level of chains is the coformality
of the induced $\End(F_n)$  action on the $\Omega$-module $C_*((\vee_n S^1)^{\wedge \bullet})$.\footnote{By this we mean that every induced map of the action is coformal, see Definition~\ref{d:coformal} and Proposition~\ref{p:coformal_susp}.}

\begin{proof}[Proof of Theorem~\ref{thm:endf_n_HP}]
At this point we only need to explain what is the Hodge splitting in $CH^{\vee_n S^1}(L_*)$,
 why it is preserved by the $\End(F_n)$ action as a filtration, and why on the associated graded complex
$\gr\, CH^{\vee_n S^1}(L_*)$ this action factors through $\End(\Z^n)$.

In case $n=1$, i.e. for the usual Hochschild complex $CH^{S^1}(\L_*)$, the Hodge splitting is obtained by noticing that the action of $\End(F_1)=(\Z,*)$ splits this complex into a direct  product 
of spaces numbered by non-negative integers, such that on the $m$-th component $r\in(\Z,*)$ acts as multiplication by $r^m$~\cite{GerstSchack,Loday}. The projection on the $m$-th component is called $m$-th
Euler idempotent $e_m$. Notice that each component $N\L_*(\ell)[\ell]$ of the complex
\[
CH^{S^1}(\L_*)=\Tot(\L_*\circ S^1_\bullet)=
\left(\prod_{\ell\geq 0} N\L_*(\ell)[\ell],d\right)
\]
is acted on by $\bbS_\ell$ and thus by the group algebra $\Q[\bbS_\ell]$. The Euler idempotent $e_m(\ell)$ is
obtained via this action and is in fact an element of $\Q[\bbS_\ell]$.  To give a bit more insight, one has 
an isomorphism of symmetric sequences:
\[
\Com\circ\Lie\stackrel{\simeq}{\longrightarrow}\Ass,
\]
induced by
the Poincar\'e-Birkhoff-Witt  map. The image of $e_m(\ell)$  is exactly
\[
\left[\Com(m)\circ\Lie\right](\ell)\subset \Ass(\ell)=\Q[\bbS_\ell].
\]
When $n\geq 2$, to obtain a similar splitting in Hochschild-Pirashvili homology one can  use the action of the monoid $(\Z,*)^{\times n}\subset\End(F_n)$ 
consisting of the homotopy classes of maps $\vee_n S^1\to \vee_n S^1$ sending each circle into itself. 
The complex $CH^{\vee_n S^1}(L_*)$ splits into a direct product of spaces numbered by $n$-tuples 
$(m_1,\ldots,m_n)$ of non-negative integers. Element $(r_1,\ldots,r_n)\in (\Z,*)^{\times n}$ acts on the
$(m_1,\ldots,m_n)$ component of the Hodge splitting as multiplication by $r_1^{m_1}\cdot\ldots\cdot
r_n^{m_n}$. Each $(\ell_1,\ldots,\ell_n)$ component $N\L_*(\ell_1+\ldots+\ell_n)$ 
of $\Tot(\L_*\circ (\vee_n S^1)_{\underbrace{\bullet\ldots\bullet}_n})$ is acted on by $\bbS_{\ell_1}\times
\ldots\times \bbS_{\ell_n}$. The projection onto the $(m_1,\ldots,m_n)$ Hodge component is given by 
$e_{m_1}(\ell_1)\otimes\ldots\otimes e_{m_n}(\ell_n)$. We define the {\it total Hodge degree} as 
$m=m_1+\ldots +m_n$. One can see that the action of $\End(F_n)$ preserves it as a filtration. 

To see that the $\End(F_n)$ action on $\gr\, CH^{\vee_n S^1}(L_*)$ factors through $\GL(n,\Z)$,
see equations~\eqref{eq:grCH_m1}, \eqref{eq:grCH_m2}, and  Remark~\ref{r:grCH_wedge_n}, which describe $\gr\, CH^{\vee_n S^1}(L_*)$  in terms of $H^1(\vee_n S^1)$. 
\end{proof}

\section{Hochschild-Pirashvili homology on suspensions. Proof of Theorem~\ref{t:HP_suspensions1}}\label{s4}

\subsection{Complexes $\gr\, \CH^{\Sigma Y_*}(\L_*) $}\label{ss41}

In this subsection we describe complexes computing higher Hochschild homology 
on suspensions $HH^{\Sigma Y_*}(\L_*) $.  These complexes depend only on $\tilde H_*(Y_*)$ and 
as we will later see in Subsection~\ref{ss44} they can be naturally identified with the associated graded of $\CH^{\Sigma Y_*}(\L_*) $.

One of the two  reasons for the Hodge splitting in the higher Hochschild homology (on a suspension) is the formality of the $\Gamma$-module $C_*(X_*^\bullet)$ in case $X_*=\Sigma Y_*$. Recall that a $\Gamma$-module is said  {\it formal} 
if it is quasi-isomorphic via a zigzag of quasi-isomorphisms to its homology $\Gamma$-module. Similarly,
a map between $\Gamma$-modules is formal if this map is quasi-isomorphic via a zigzag of quasi-isomorphisms of $\Gamma$-modules maps to the induced map in their homology.

\begin{lemma}\label{l:formality1}
If a pointed space $X_*$  is of finite type and is rationally formal, then the right $\Gamma$ module 
$C_*(X_*^\bullet)$  is also rationally formal.
If a pointed map $X_*\to Y_*$ between spaces of finite type is rationally formal, then the induced
map of $\Gamma$ modules $C_*(X_*^\bullet) \to C_*(Y_*^\bullet)$ is also formal.
\end{lemma}

\begin{proof}
By formality of a space we understand formality of its Sullivan algebra $A_{X_*}$ as augmented algebra
and similarly for a map between spaces. We show explicitly the first statement. The second one follows from functoriality of the construction. One has a quasi-isomorphism of $\Gamma$-modules:
\[
C_*(X_*^\bullet)\simeq \left(A_{X_*^\bullet}\right)^\vee\simeq
\left(\Q\otimes A_{X_*}^{\otimes \bullet}\right)^\vee
\simeq \left(\Q\otimes H^*(X_*)^{\otimes \bullet}\right)^\vee
\simeq H_*(X_*^\bullet).
\]
\end{proof}

\begin{lemma}\label{l:formality2}
Any suspension of a space of finite type is rationally formal and, moreover, any suspension
of a map between spaces of finite type is rationally formal.
\end{lemma}

 Recall that a map of pointed spaces is formal if the induced map of Sullivan augmented algebras is formal, i.e.,  quasi-isomorphic to the 
map of rational cohomology algebras (in the category of augmented algebras).  In particular it implies that each space is formal. 

\begin{proof}
Let $Y_*$ be a space of finite type and let us show that $\Sigma Y_*$ is formal. The argument for 
a map between suspensions 
is similar. In case $Y_*$ 
is connected, its suspension $\Sigma Y_*$ is simply connected. It is also a co-$H$-space, therefore it is coformal and its Quillen model is a free Lie algebra generated by $\tilde H_*(Y_*)$ with zero differential. 
The Koszul dual commutative algebra is generated by $\tilde H^*(\Sigma Y_*)$ with all products of generators being zero. 

In case $Y_*=\coprod_{i=1}^k Y_i$ is a disjoint union of $k$ components, then 
$\Sigma Y_*=\left(\bigvee_{k-1}S^1\right)\vee\left(\bigvee_{i=1}^k\Sigma Y_i\right).
$
And the wedge of formal spaces is formal.
\end{proof}

Notice that from these two lemmas it follows that if $X_*$  is a suspension of finite type, then $C_*(X_*^\bullet)$ is a formal $\Gamma$-module and that the same is true for  a suspension of a map between spaces of finite type.  Proposition~\ref{p:formal_susp}  below implies that the finiteness condition can be released.

%

Let $\Omega$ be the category of finite sets with morphisms all surjective maps. In~\cite{pira00} 
Pirashvili defines an equivalence of categories
\[
cr\colon \Mod{-}\Gamma\to \Mod{-}\Omega.
\]
On objects
\begin{equation}\label{eq_cr}
cr\,\L_*(k)=\L_*(k_*)\Big/ +_{i=1}^k {\mathrm {Im}}\, r_i^*,
\end{equation}
where $r_i^*\colon \L_*(k_*\setminus\{i\})\to \L_*(k_*)$ is induced by the map
$r_i\colon k_*\to k_*\setminus \{i\}$:
\[
r_i(j)=
\begin{cases}
j,& j\neq i;\\
*,& j=i.
\end{cases}
\]

On morphisms $cr\, \L_* $ is obtained as restriction with respect to the inclusion $i\colon \Omega\to\Gamma$ that adds the basepoint to any set: $i(k)=k_*$. Recall~\eref{eq_NL}. The space $cr\,\L_*(k)$
is isomorphic to $N\L_*(k)$ via the obvious composition
\begin{equation}\label{eq:q}
q\colon N\L_*(k)\hookrightarrow\L_*(k_*)\to cr\,\L_*(k).
\end{equation}
One can show that $q$ is an isomorphism using the map $\prod_{i=1}^k(1-r_i^*s_i^*)$ that projects 
$\L_*(k_*)$ onto $N\L_*(k)$. (Notice that $r_i^*s_i^*$, $i=1\ldots k$, are pairwise commuting projectors
as well as $(1-r_i^*s_i^*)$, $i=1\ldots k$.) For the complexes that we consider below it is sometimes convenient to use $N\L_*(\bullet)$
instead of $cr\,\L_*(\bullet)$.

  Let us  describe the induced $\Omega$-module structure on $N\L_*(\bullet)$. The symmetric group action
  as part of $\Omega$ structure on $N\L_*(\bullet)$ is the usual one. Denote by $m_i\colon (k+1)\to
  k$ the surjection
  $$
  m_i(j)=\begin{cases}
  j,& 1\leq j\leq i;\\
  j-1,& i+1\leq j\leq k.
  \end{cases}
  $$
  Abusing notation we denote by $m_i\colon (k+1)_*\to k_*$ the same map extended as $m_i\colon *\mapsto *$. For $\gamma\in cr\,\L_*(k)$, one has
  \begin{equation}\label{eq_right_inf_bim0}
  q^{-1}(m_i^*(\gamma))=(1-r_i^*s_i^*-r_{i+1}^*s_{i+1}^*)m_i^* (q^{-1}(\gamma)).
  \end{equation}
  One can write this formula slightly differently. Recall that the structure of a right $\Omega$-module is equivalent to the structure of a right module over the commutative non-unital operad $\Com_+$, while the structure of a right $\Gamma$-module is equivalent to the structure of an infinitesimal bimodule over the commutative unital operad $\Com$, see~\cite[Proposition~4.9]{AroneTur2} or~\cite[Lemma~4.3]{Turchin1}. 
In this terms, equation~\eqref{eq_right_inf_bim0} is written as
\begin{multline}\label{eq_right_inf_bim}
q^{-1}\bigl(\gamma(x_1,\ldots,x_i\cdot x_{i+1},\ldots,x_{k+1})\bigr)=
q^{-1}(\gamma)(x_1,\ldots,x_i\cdot x_{i+1},\ldots,x_{k+1}) \\
-
x_i\cdot q^{-1}(\gamma)(x_1\ldots \hat x_i \ldots x_{k+1}) 
- x_{i+1}\cdot q^{-1}(\gamma)(x_1\ldots \hat x_{i+1} \ldots x_{k+1}).
\end{multline}
The two last summands in~\eqref{eq_right_inf_bim0} and~\eqref{eq_right_inf_bim} are correction terms necessary to make the right-hand side normalized.

The higher Hochschild homology over a pointed space $X_*$ is computed as the space of homotopy maps of $\Gamma$-modules
\[
HH^{X_*}(\L_*)=
H_*\bigl( \hRmod_\Gamma\left(C_*(X_*^\bullet),\L_*\right)\bigr).
\]

For any pointed space $X_*$, the cross-effect of the $\Gamma$-module $C_*(X_*^\bullet)$ is equivalent to 
\begin{equation}\label{eq:cross}
cr\, C_*(X_*^\bullet)\simeq \tilde C_*(X_*^{\wedge\bullet}),
\end{equation}
see~\cite{AroneTur1}, where the $\Omega$-module structure on $\tilde C_*(X_*^{\wedge\bullet})$ 
is induced by the diagonal maps. For any surjection $p\colon k\twoheadrightarrow \ell$, one gets a map $X_*^{\wedge\ell}\to X_*^{\wedge k}$ defined as 
\begin{equation}\label{eq:diagonal_map}
(x_1,\ldots x_\ell)\mapsto (x_{p^{-1}(1)},\ldots,x_{p^{-1}(k)}).
\end{equation}

It follows that the Hochschild-Pirashvili homology can also be described as
\[
HH^{X_*}(\L_*)=
H_*\left( \hRmod_\Omega\left(\tilde C_*(X_*^{\wedge\bullet}),cr\, \L_*\right)\right).
\]

\begin{defi}\label{d:omega_trivial}
We say that a right $\Omega$ module $M$ has a trivial $\Omega$ action if for any strict surjection 
$p\colon k\twoheadrightarrow \ell$ the induced
map  $M(\ell)\to M(k)$ is the zero map.
\end{defi}
%
%

\begin{prop}\label{p:formal_susp}
For any pointed suspension $\Sigma Y_*$, the $\Omega$ module $\tilde C_*\left((\Sigma Y_*)^{\wedge\bullet}\right)$ is formal. For any pointed map $g\colon Y_*\to Z_*$,
the induced map of $\Omega$ modules $(\Sigma g)_*\colon \tilde C_*\left((\Sigma Y_*)^{\wedge\bullet}\right)\to \tilde C_*\left((\Sigma Z_*)^{\wedge\bullet}\right)$ is also formal.
\end{prop}

\begin{proof}

For the proof we will need that the $\Omega$ module $ \tilde C_*((S^1)^{\wedge\bullet}) $ is formal and has the trivial $\Omega$ action in homology.
 The first statement follows from the fact that the $\Gamma$ module
$C_*((S^1)^\bullet)$ is formal (by Lemmas~\ref{l:formality1} and~\ref{l:formality2})
and thus is so its cross-effect $cr\, C_*((S^1)^\bullet) \simeq \tilde C_*((S^1)^{\wedge\bullet}) $. 
The second statement is straightforward as any diagonal map $S^\ell\to S^k$ for $k>\ell$ induces 
the zero map in reduced homology.

The following sequence of quasi-isomorphisms of $\Omega$ modules proves the formality of $\tilde C_*\left((\Sigma Y_*)^{\wedge\bullet}\right)$:\footnote{This simple argument  was provided to us by G.~Arone.}
\begin{multline}\label{eq:formality_arone}
\tilde C_*\left((\Sigma Y_*)^{\wedge\bullet}\right)\simeq \tilde C_*\left((S^1)^{\wedge\bullet}\right)\otimes \tilde C_* \left( Y_*^{\wedge\bullet}\right)\simeq
 \tilde H_*\left((S^1)^{\wedge\bullet}\right)\otimes \tilde C_* \left( Y_*^{\wedge\bullet}\right)\simeq \\
  \simeq\tilde H_*\left((S^1)^{\wedge\bullet}\right)\otimes \tilde C_*( Y_*)^{\otimes\bullet}\simeq
     \tilde H_*\left((S^1)^{\wedge\bullet}\right)\otimes \tilde H_*( Y_*)^{\otimes\bullet}.
\end{multline}
   By the tensor product above we understand an objectwise tensor product of right $\Omega$ modules.
   The second quasi-isomorphism uses the formality of $ \tilde C_*((S^1)^{\wedge\bullet}) $. 
 Notice that all the terms in this zigzag starting from the third one have the trivial $\Omega$ action. Notice also that all the quasi-isomorphisms are functorial in $Y_*$ except the last one, which uses a choice of a quasi-isomorphism $\tilde H_*Y_* \to \tilde C_*Y_*$. On the other hand,  any  morphism of complexes (in our case  $\tilde C_*Y_*\to \tilde C_*Z_*$) is formal (i.e., is quasi-isomorphic to the induced map $\tilde H_*Y\to\tilde H_*Z$).
    This proves the formality of the induced map of $\Omega$ modules.
%

\end{proof}

\begin{rem}\label{r:triv_susp}
It follows from~\eqref{eq:formality_arone} that for any suspension $\Sigma Y_*$,  the
right $\Omega$ module $\tilde C_*\left((\Sigma Y_*)^{\wedge\bullet}\right)$ has the trivial 
$\Omega$ action in homology.
\end{rem}

This property is in fact the second of the two reasons for the Hodge splitting. (The first one is the formality.) Indeed,   
%
as a consequence, the $\Omega$-module $\tilde H_*\left((\Sigma Y_*)^{\wedge \bullet}\right)$ splits into a direct sum of $\Omega$-modules:
\begin{equation}\label{eq_suspens_split}
cr\, \tilde C_*\left((\Sigma Y_*)^{\wedge \bullet}\right) \simeq 
\tilde H_*\left((\Sigma Y_*)^{\wedge \bullet}\right) \simeq
\bigoplus_{m\geq 0}
\tilde H_*(\Sigma Y_*)^{\otimes m},
\end{equation}
where $\tilde H_*(\Sigma Y_*)^{\otimes m}$ denotes the $\Omega$-module which is 
$\tilde H_*(\Sigma Y_*)^{\otimes m}$ in arity $m$ and 0 in all others. Thus we get
\begin{equation}\label{eq_HH_susp}
HH^{\Sigma Y_*}(\L_*)\simeq \prod_{m\geq 0} H\left( \hRmod_\Omega
\left(\tilde H_*(\Sigma Y_*)^{\otimes m}, cr\,\L_*\right)\right).
\end{equation}

As a corollary we see that the functor $HH^{(-)}(\L_*)$ factors through
the reduced homology functor $\tilde H_*\colon \TOP_*\to gVect$ when restricted on $\Sigma(Top_*)$. 
The splitting by $m$ in~\eqref{eq_HH_susp} is exactly the Hodge splitting.

Now we want to make more explicit the right-hand side of~\eqref{eq_HH_susp}. Recall that the right
$\Omega$-module is the same as the right $\Com_+$-module. Applying the Koszul duality between the $\Lie$ and $\Com_+$ operads, the cofibrant replacement of $\tilde H_*(\Sigma Y_*)^{\otimes m}$
as a right $\Com_+$-module is $\tilde H_*(\Sigma Y_*)^{\otimes m}\circ\coLie\{1\}\circ\Com_+$, where $\circ$ is the composition product of symmetric sequences; $\coLie$ is the Lie cooperad; $\{1\}$ denotes 
operadic suspension~\cite{Fresse1,AroneTur2,SongTur}. The differential in it is obtained by taking off one cobracket 
from the $\coLie\{1\}$ factor and by making it act from the left on the $\Com_+$ part as a product $x_1\cdot x_2$, see~\cite[Section~5]{AroneTur2}. For a general right $\Com_+$-module $M$, there is another term of the differential on its cofibrant replacement $M\circ\coLie\{1\}\circ\Com_+$, which takes off one cobracket from the $\coLie\{1\}$ part and makes it act from the right on $M$ also as a product $x_1\cdot x_2$. But in our 
case this action is trivial, so only the first part of the differential is present.  The product over $m\geq 0$
of the complexes below computes $HH^{\Sigma Y_*}(\L_*)$:
\begin{multline}\label{eq:grCH_m1}
\Rmod_{\Com_+}\left(\tilde H_*(\Sigma Y_*)^{\otimes m}\circ\coLie\{1\}\circ\Com_+, cr\,\L_*\right)=
\left(\Hom_\bbS\Bigl(\tilde H_*(\Sigma Y_*)^{\otimes m}\circ\coLie\{1\},  cr\,\L_*\Bigr),d\right)=
\\
\Hom_{\bbS_m}\left( \tilde H_*(Y_*)^{\otimes m},\left(\prod_{\ell\geq m}\,\,\bigoplus_{\ell_1+\ldots +\ell_m=\ell}
\left(\Lie(\ell_1)\otimes\ldots\otimes\Lie(\ell_m)
\otimes_{\bbS_{\ell_1}\times\ldots\times \bbS_{\ell_m}}\Bigl(sign\otimes cr\,\L_*(\ell)\Bigr)\right)[\ell], d\right)\right),
\end{multline}
which assuming the finiteness condition on the homology of $Y_*$  can also be
written  as
\begin{equation}\label{eq:grCH_m2}
\tilde H^*(Y_*)^{\otimes m}\hat\otimes_{\bbS_m}
\left(\prod_{\ell\geq m}\,\,\bigoplus_{\ell_1+\ldots +\ell_m=\ell}
\left(\Lie(\ell_1)\otimes\ldots\otimes\Lie(\ell_m)
\otimes_{\bbS_{\ell_1}\times\ldots\times \bbS_{\ell_m}}\Bigl(sign\otimes cr\,\L_*(\ell)\Bigr)\right)[\ell]
,d\right).
\end{equation}
Here $sign$ denotes the sign representation of $\bbS_\ell$; the reduced cohomology of $Y_*$ is viewed as a negatively graded vector space.
The differential in this complex is the sum of simultaneous insertions of $[x_1,x_2]$ in one of the inputs of 
$\Lie(\ell_i)$ for some $i$, and right action by $x_1\cdot x_2$ on the corresponding input of $cr\,\L_*(\ell)$.
Beware that if we replace $cr\,\L_*(\ell)$ by $N\L_*(\ell)$ additional summands in the differential appear due 
to the last two terms in~\eqref{eq_right_inf_bim0}-\eqref{eq_right_inf_bim}. 

\begin{rem}\label{r:grCH_C_M}
In case $Y_*$ is of finite type, and $\L_*=M\otimes C^{\otimes\bullet}$,
the obtained complex computing $HH^{\Sigma Y_*}(C,M)$ is 
\begin{equation}\label{eq:grCH_C_M}
M\hat\otimes S\left(\tilde H^*(Y_*)\hat\otimes{\mathcal L}(C)\right),
\end{equation}
where the cohomology $\tilde H^*(Y_*)$ is non-positively graded; ${\mathcal L}(C)$ is the Harrison complex
of $C$. The symmetric power and tensor products are the completed ones. 
 The differential
\[
d=d_M+d_C+\delta,
\]
where $d_M$ and $d_C$ are induced by the differential on $M$ and ${\mathcal L}(C)$, and 
$\delta(m\otimes x)=m'\otimes [m'',x]$. The part $\delta$ in the differential appears due to the last two summands in~\eqref{eq_right_inf_bim0}-\eqref{eq_right_inf_bim}.\footnote{To recall $C$ is simply connected. If $C$ is not simply connected,  the Harrison complex  ${\mathcal L}(C)$ should be replaced 
by the completed Harrison complex $\hat {\mathcal L}(C)$.}

\end{rem}

\begin{rem}\label{r:grCH_wedge_n}
For $Y_*=\vee_n S^0$ and any $\L_*$, the obtained complex is identical to 
$\gr\, CH_*^{\vee_n S^1}(\L_*)$ considered in Section~\ref{s:CH_vee_n}.
In case $\L_*=M\otimes C^{\otimes\bullet}$ it follows from Proposition~\ref{p:out_hodge_filtr} and Remark~\ref{r:grCH_C_M}.
For a general $\L_*$ one can construct this isomorphism analogously. The idea is that elements of $\Lie(\ell_i)$ in~\eqref{eq:grCH_m2} should be viewed as linear combinations of permutations in $\bbS_{\ell_i}$,
which tells us in which order the elements should be put on the corresponding circle.
\end{rem}

\subsection{Hodge filtration. Proof of Theorem~\ref{t:HP_suspensions1}}\label{ss42}
We define a functorial filtration on the space of homotopy maps of right $\Omega$-modules, which induces
the desired filtration on $HH^{X_*}(\L_*)$ functorial in $X_*$ and $\L_*$.  For a right $\Omega$-module
$K$ define its $m$-th truncation $tr_m K$ as 
\[
tr_m(K)(\ell) =
\begin{cases}
K(\ell),& \ell\leq m;\\
0,& \ell>m.
\end{cases}
\]
This symmetric sequence has an obvious $\Omega$-module structure, such that the projection
$K\to tr_m K$ is an $\Omega$-modules map. This morphism for any $\Omega$-module $L$ induces a map of complexes
\[
\hRmod_{\Omega}(tr_m K,L)\to \hRmod_{\Omega}( K,L).
\]
Its image in homology is what we call the $m$-th term of the Hodge filtration in
$H\left(\hRmod_{\Omega}( K,L\right))$.

For $K=\tilde C_*\left( (\Sigma Y_*)^{\wedge\bullet}\right)\simeq \tilde H_*(\Sigma Y_*)^{\otimes\bullet}$,
the cofiltration $tr_\bullet$ splits. For any pointed map of suspensions $\Sigma Y_*\to \Sigma Z_*$,
the induced map 
\[
\gr\, HH^{\Sigma Z_*}(\L_*)\to \gr\, HH^{\Sigma Y_*}(\L_*)
\]
can be recovered from the map of the layers of $tr_\bullet$ (and thus from the map
in homology $\tilde H_*\Sigma Y_*\to \tilde H_*\Sigma Z_*$) by the   spectral sequence argument.

\subsection{Hodge filtration versus cardinality cofiltration}\label{ss42+}
Denote by $\CHH^{X_*}(\L_*)$ the higher Hochschild complex 
$$
\CHH^{X_*}(\L_*):=\hRmod_\Omega \left(  \tilde C_* \left( X_*^{\wedge\bullet}\right), cr\,\L_*\right).
$$
The Hodge filtration
\[
F_0 \CHH^{X_*}(\L_*) \to F_1 \CHH^{X_*}(\L_*) \to F_2 \CHH^{X_*}(\L_*) \to \ldots
\]
should not be confused with the more widely used cardinality or rank cofiltration (depending on the context 
it can also be called Goodwillie-Weiss tower)~\cite{AyFr,IntJMc,WeissEmb}:
\[
T_0 \CHH^{X_*}(\L_*) \leftarrow T_1 \CHH^{X_*}(\L_*) \leftarrow T_2 \CHH^{X_*}(\L_*) \leftarrow \ldots .
\]
We have seen in the previous subsection that
\[
F_m \CHH^{X_*}(\L_*) \simeq \hRmod_\Omega \left( tr_m \tilde C_* \left( X_*^{\wedge\bullet}\right),cr\,\L_*\right).
\]

\begin{prop}\label{p:cardinal}
The $n$-th term of the cardinality cofiltration is
\[
T_m \CHH^{X_*}(\L_*) \simeq \hRmod_\Omega \left(  \tilde C_* \left( X_*^{\wedge\bullet}\right),tr_m cr\,\L_*\right).
\]
\end{prop}

\begin{proof}
Denote by $\Gamma_m$ and $\Omega_m$ the full subcategories of $\Gamma$, respectively $\Omega$, 
consisting of objects of cardinal $\leq m+1$, respectively $\leq m$. One has obvious restriction functors
\[
(-)|_{\leq m}\colon \Mod{-}\Gamma\to \Mod{-}\Gamma_m;\qquad
(-)|_{\leq m}\colon \Mod{-}\Omega\to \Mod{-}\Omega_m.
\]

By definition
\begin{equation}\label{eq:cardinal}
T_m \CHH^{X_*}(\L_*) \simeq \hRmod_{\Gamma_m} \left(  C_* \left( X_*^{\bullet}\right)|_{\leq m}, \L_*|_{\leq m}\right).
\end{equation}

The cross-effect functor
\[
cr\colon\Mod{-}\Gamma_m\to \Mod{-}\Omega_m
\]
defined by~\eqref{eq_cr} is also an equivalence in the truncated case.

For a right $\Omega_m$ module $K$, denote by $triv_m (K)$ the $\Omega$ module extended trivially on sets of cardinal $>m$:
\[
triv_m(K)(\ell) =
\begin{cases}
K(\ell),& \ell\leq m;\\
0,& \ell>m.
\end{cases}
\]
One has a Quillen adjunction
\[
(-)|_{\leq m}\colon\Mod{-}\Omega\rightleftarrows \Mod{-}\Omega_m\colon triv_m.
\]
Notice that $triv_m\circ (-)_{\leq m} = tr_m$. As a consequence we get
\[
T_m \CHH^{X_*}(\L_*) \simeq \hRmod_{\Omega_m} \left(  \tilde C_* \left( X_*^{\wedge\bullet}\right)|_{\leq m}, cr\,\L_*|_{\leq m}\right)\simeq \hRmod_\Omega \left(  \tilde C_* \left( X_*^{\wedge\bullet}\right), tr_m cr\,\L_*\right).
\]
\end{proof}

Finally, let us compare the $T_m$ and $F_m$ terms in case of a suspension to make sure that they are different.
\begin{gather*}
F_m \CH^{\Sigma Y_*}(\L_*) =
\prod_{i=0}^m\hRmod_\Omega\left( \tilde H_*(\Sigma Y_*)^{\otimes i},cr\, \L_*\right)=
\prod_{i=0}^m\left(\prod_{j=i}^{+\infty} \Hom_{\bbS_j}\left( ( \tilde H_*(\Sigma Y_*)^{\otimes i}\circ
\coLie\{1\} )(j), cr\, \L_* (j) \right), d\right);\\
T_m \CH^{\Sigma Y_*}(\L_*) =
\prod_{i=0}^{+\infty}\hRmod_\Omega\left( \tilde H_*(\Sigma Y_*)^{\otimes i}, tr_m cr\, \L_*\right)=
\prod_{i=0}^m\left(\prod_{j=i}^{m} \Hom_{\bbS_j}\left( ( \tilde H_*(\Sigma Y_*)^{\otimes i}\circ
\coLie\{1\} )(j), cr\, \L_* (j) \right), d\right).
\end{gather*}
One can see that the terms $F_m$ and $T_m$ are not the same.

\begin{rem}\label{r:cardin}
The cardinality cofiltration induces a decreasing filtration in $\CH^{\Sigma Y_*}(\L_*)$: we define
$F^m \CH^{\Sigma Y_*}(\L_*)$ as the kernel of the projection $p_m\colon\CH^{\Sigma Y_*}(\L_*)\to T_m \CH^{\Sigma Y_*}(\L_*)$. Notice that $p_m$ restricted on $F_m \CH^{\Sigma Y_*}(\L_*)$ is still surjective.
As a consequence, one has that the Hodge filtration in the Hochschild-Pirashvili homology on a suspension  is dense in the topology induced by this decreasing filtration. 
\end{rem}

\begin{rem}\label{r:cardin2}
The cardinality cofiltration in the higher Hochschild homology on suspensions, contrary to the Hodge filtration, does not split in general.
\end{rem}

\section{Coformality of $C_*\left( (\Sigma Y_*)^{\wedge\bullet} \right)$. Proof of Theorem~\ref{t:HP_suspensions2}}\label{ss43}
We need to recall some theory of right modules over $\Com_+$~\cite{Fresse1}. As we briefly explained in 
Subsection~\ref{ss41}, a functorial cofibrant replacement of a right $\Omega$-module or
equivalently a right $\Com_+$-module $M$
is $M\circ\coLie\{1\}\circ\Com_+$.  The sequence $M\circ\coLie\{1\}$ is the {\it Koszul dual} of $M$. Notice that it is naturally a right $\coLie\{1\}$-comodule. Given any other right  $\coLie\{1\}$-comodule $N$, one can get a $\Com_+$-module $N\circ\Com_+$.\footnote{The differential in $N\circ\Com_+$ is the sum of two terms: the first one being induced by the differential on $N$, the second splits off one cobracket from $N$
and makes it act from the left as a product on $\Com_+$.}  It is easy to see that $N\circ \Com_+$ is quasi-isomorphic to $M$ (as a $\Com_+$-module) if and only if $N$ is quasi-isomorphic to $M\circ\coLie\{1\}$
(as a $\coLie\{1\}$-comodule). If this happens we say that $N$ is a Koszul dual of $M$ and $M$ is a Koszul dual of $N$. 

This is part of a general homotopy theory of right modules~\cite{Fresse1}. For any right module $M$
over any doubly reduced operad ${\mathcal O}$ in chain complexes (${\mathcal O}(0)=0$, ${\mathcal O}(1)=\Q$), the bar construction $B(M,{\mathcal O},I)$ is a right comodule over the cooperad $B(I,{\mathcal O},I)$. By $I$ we mean the unit object in symmetric sequences 
$$
I(k)=
\begin{cases}
\Q,& k=1;\\
0,& k\neq 1.
\end{cases}
$$
 In our case the operad ${\mathcal O}=\Com_+$ is Koszul and the bar complexes can be replaced by equivalent Koszul complexes~\cite{Fresse1}.

 It was shown by~\cite[Lemma~11.4]{AroneTur1}, that for any pointed space $X_*$, the Koszul dual of 
$\tilde C_*(X_*^{\wedge\bullet})$ is 
$\tilde C_*(X_*^{\wedge\bullet}/\Delta^\bullet X_*)$, where by $\Delta^n X_*$
we understand the fat diagonal in $X_*^{\wedge n}$. On homology the $\coLie\{1\}$ coaction
\[
\circ_{i\sim j}\colon\tilde H_*(X_*^{\wedge n}/\Delta^n X_*)\to \tilde H_{*-1}(X_*^{\wedge n-1}/\Delta^{n-1} X_*)
\otimes \coLie\{1\}(2)
\]
is induced by the connecting homomorphisms 
$\partial\colon H_*(X_*^{\wedge n},\Delta^n X_*)\to H_{*-1}(\Delta^n X_*, \Delta_{ij}^n X_*)$
  of the long exact sequence for the triples
\[
(X_*^{\wedge n},\Delta^n X_*,\Delta_{ij}^n X_*),
\]
where $\Delta_{ij}^n X_*$ is the union of all diagonals except one: $x_i=x_j$. (One obviously has
$\Delta^n X_*/\Delta_{ij}^n X_*\cong X_*^{\wedge n-1}/\Delta^{n-1} X_*$.)

\begin{defi}\label{d:coformal}
We say that a right $\Com_+$-module is coformal if its Koszul dual $\coLie\{1\}$-comodule is formal.
A map of right $\Com_+$-modules is said coformal if the induced morphism of their Koszul duals is formal.
\end{defi}

\begin{prop}\label{p:coformal_susp}
\sloppy
For any pointed suspension $\Sigma Y_*$, the right $\Com_+$-module $\tilde C_*\left((\Sigma Y_*)^{\wedge\bullet}\right)$ is coformal.
For any pointed map of suspensions $f\colon\Sigma Y_*\to\Sigma Z_*$, the induced map of
$\Com_+$-modules $f_*\colon \tilde C_*\left((\Sigma Y_*)^{\wedge\bullet}\right)\to \tilde C_*\left((\Sigma Z_*)^{\wedge\bullet}\right)$  is coformal.  
\end{prop}
%

\begin{proof} 
According to Proposition~\ref{p:formal_susp} both $\Com_+$-modules $\tilde C_*\left((\Sigma Y_*)^{\wedge\bullet}\right)$ and $\tilde C_*\left((\Sigma Z_*)^{\wedge\bullet}\right)$ are formal. Their Koszul duals
are $\tilde H_*(\Sigma Y_*)^{\otimes \bullet}\circ \coLie\{1\}$ and $\tilde H_*(\Sigma Z_*)^{\otimes \bullet}\circ \coLie\{1\}$, see Subsection~\ref{ss41}, which are formal and cofree. On the other hand 
it is easy to see that any map between right $\coLie\{1\}$-comodules whose homology is cofree, is formal.
\end{proof}

\begin{cor}\label{cor:conf_susp}
One has a natural isomorphism of right $\coLie\{1\}$-comodules 
\begin{equation}\label{eq_PBW_general}
\tilde H_* \left((\Sigma Y_*)^{\wedge\bullet}/\Delta^\bullet Y_*\right){\stackrel \simeq\longrightarrow} \tilde H_*(\Sigma Y_*)^{\otimes\bullet}\circ\coLie\{1\},
\end{equation}
functorial over the category $\Sigma(\TOP_*)$.
\end{cor}
One simply needs to apply the Koszul duality functor to the zigzag~\eqref{eq:formality_arone} and then take the homology. At the starting point we get the left-hand side of~\eqref{eq_PBW_general} and at the end we get the right-hand side. Notice that this corollary describes the rational homology of certain configuration spaces of points in suspensions. 

Now notice that the sequences $\tilde H_*(\Sigma Y_*)^{\otimes\bullet}$ and $\tilde H_* \left((\Sigma Y_*)^{\wedge\bullet}/\Delta^\bullet \Sigma Y_*\right)$ are naturally left modules over the 
commutative operad $\Com$. Indeed, the first one is  freely generated by its arity one
    component $\tilde H_*(\Sigma Y_*)^{\otimes 1}$, while the left $\Com$-module structure on the second one is induced
    by the maps
    \[
    \left( (\Sigma Y_*)^{\wedge m}/\Delta^m \Sigma Y_*\right)\wedge
    \left( (\Sigma Y_*)^{\wedge n}/\Delta^n \Sigma Y_*\right)\longrightarrow
    \left( (\Sigma Y_*)^{\wedge m+n}/\Delta^{m+n} \Sigma Y_*\right)
    \]
    (More generally if a right $\Com_+$-module has a compatible left action by another operad $\mathcal O$, then its Koszul dual also naturally is a left $\mathcal O$-module.)
    
 \begin{prop}\label{pr:left_act}
    The isomorphism~\eqref{eq_PBW_general} respects the left $\Com$ action.
 \end{prop}
 \begin{proof}
It is enough to check that each map in the zigzag~\eqref{eq:formality_arone} respects the left $\Com$ action.
\end{proof}

\subsection{Complexes $\CH^{\Sigma Y_*}(\L_*)$. Proof of Theorem~\ref{t:HP_suspensions2}}\label{ss44}
\sloppy
We define complexes $\CH^{\Sigma Y_*}(\L_*)$ as follows
\begin{multline}\label{eq_CH_general}
\Rmod_{\Com_+}\left( 
 \tilde H_* \left((\Sigma Y_*)^{\wedge\bullet}/\Delta^\bullet \Sigma Y_*\right)\circ\Com_+, cr\,\L_*\right)\simeq\\
 \left( \prod_{n\geq 0} \Hom_{\bbS_n}\left( 
 \tilde H_* \left((\Sigma Y_*)^{\wedge n}/\Delta^n \Sigma Y_*\right), cr\,\L_*(n)\right), d_{Y_*}+d_{\L_*}\right),
 \end{multline}
 where $d_{\L_*}$ is the part of the differential induced by the differential in $\L_*$, and $d_{Y_*}$ 
 is induced by the differential in $\tilde H_* \left((\Sigma Y_*)^{\wedge\bullet}/\Delta^\bullet \Sigma Y_*\right)\circ\Com_+$, which is the Koszul dual $\Com_+$-module to the $\coLie\{1\}$-comodule
 $\tilde H_* \left((\Sigma Y_*)^{\wedge\bullet}/\Delta^\bullet \Sigma Y_*\right)$. Explicitly, if $f\in \Hom_{\bbS_n}\left( 
 \tilde H_* \left((\Sigma Y_*)^{\wedge n}/\Delta^n \Sigma Y_*\right), cr\,\L_*(n)\right)$, one has 
 $d_{Y_*}f\in \Hom_{\bbS_{n+1}}\left( 
 \tilde H_* \left((\Sigma Y_*)^{\wedge n+1}/\Delta^{n+1} \Sigma Y_*\right), cr\,\L_*(n+1)\right)$ is defined
 as follows
 \[
 (d_{Y_*}f)\bigl( \gamma(x_1\ldots x_{n+1})\bigr) = \sum_{1\leq i<j\leq n}f( \gamma_{ij}(x_1\ldots x_{i\sim j}\ldots x_n))\circ_{i\sim j}(x_i\cdot x_j),
 \]
 where $\gamma_{ij}$ is computed from the formula $\circ_{i\sim j} (\gamma)=\gamma_{ij}\otimes [x_i,x_j]^\vee$ of  the $\coLie\{1\}$ coaction.
 
 Now we check that $\CH^{(-)}(\L_*)$ satisfies the properties from Theorem~\ref{t:HP_suspensions2}.
 Firstly, $\CH^{(-)}(\L_*)\colon\TOP_*|_{\Sigma}\to dgVect$ is a well defined functor:
 a pointed map $\Sigma Y_*\to \Sigma Z_*$ induces a map of $\coLie\{1\}$-comodules
 \[
   \tilde H_* ((\Sigma Y_*)^{\wedge\bullet}/\Delta^\bullet \Sigma Y_*) \to
    \tilde H_* ((\Sigma Z_*)^{\wedge\bullet}/\Delta^\bullet \Sigma Z_*).
    \]
 It computes the Hochschild-Pirashvili homology functor by the coformality property, see Proposition~\ref{p:coformal_susp}. Using isomorphism~\eqref{eq_PBW_general} we can define the $m$-th truncation of
    $\tilde H_* \left((\Sigma Y_*)^{\wedge\bullet}/\Delta^\bullet \Sigma Y_*\right)$ 
    as the cofree part cogenerated by $\tilde H_*(\Sigma Y_*)^{\otimes i}$, $i\leq m$. In the Hochschild homology this obviously corresponds to the Hodge filtration defined in Subsection~\ref{ss42}. The map of
    graded quotients is determined by the morphism in homology $f_*\colon\tilde H_*(\Sigma Y)\to \tilde H_*(\Sigma Z)$ due to Corollary~\ref{cor:conf_susp} and Proposition~\ref{pr:left_act} (see also next section,
    where this is shown more explicitly). 
    The splitting of the Hodge filtration over $\Sigma(\TOP_*)$ has been shown in the previous section. 
    
    \sloppy
   Now let us check that the complexes $\CH^{\vee_n S^1}(\L_*)$ coincide with $CH^{\vee_n S^1}(\L_*)$
   defined in Section~\ref{s:CH_vee_n}.  To see this one needs to identify $cr\,\L_*(\bullet)$ with  $N\L_*(\bullet)$ by means of the isomorphism~\eqref{eq:q}. For simplicity let us start with the case $n=1$.
   One has $(S^1)^{\wedge k}/\Delta^k S^1=\vee_{k!}S^k$. Thus,
   \[
   \prod_{k\geq 0} \Hom_{\bbS_k}\left( 
 \tilde H_* \left((S^1)^{\wedge k}/\Delta^k S^1 \right), N\L_*(k)\right)=
 \prod_{k=0}^{+\infty} N\L_*(k)[k] =\Tot \,\L_*\circ (S^1)_\bullet.
 \]
 One can check that the differentials agree. In case of arbitrary $n$, one has $\left(\vee_n S^1\right)^{\wedge k}
 /\Delta^k(\vee_n S^1)=\vee_{k_1+\ldots +k_n=k}\vee_{k!} S^k$, and one similarly gets
 \[
   \prod_{k\geq 0} \Hom_{\bbS_k}\left( 
 \tilde H_* \left((\vee_n S^1)^{\wedge k}/\Delta^k (\vee_n S^1) \right), N\L_*(k)\right)=
 \prod_{k=0}^{+\infty} \,\prod_{k_1+\ldots +k_n=k} N\L_*(k)[k] =\Tot (\,\L_*\circ (\vee_n S^1)_{\underbrace{\bullet\ldots\bullet}_n}).
 \]
   For the last identity, see equation~\eqref{eq_CH_vee_n}.
   
   \begin{rem}\label{rem:action_ident}
   The monoid $\End(F_n)$ describes the homotopy classes of poined 
   self-maps $\vee_n S^1\to\vee_n S^1$ and thus acts on the $\coLie\{1\}$-comodule
   $ \tilde H_* \left((\vee_n S^1)^{\wedge \bullet}/\Delta^\bullet (\vee_n S^1) \right)$. One can check
   that the induced action on $\CH^{\vee_n S^1}(\L_*)$ coincides with the one on $CH^{\vee_n S^1}(\L_*)$
   described explicitly in Section~\ref{s:CH_vee_n}.
   \end{rem}
   
\section{Determining the map of Hochschild-Pirashvili homology from the rational homotopy type of a map}\label{s:rht_map}
\sloppy
    It is clear from the definition that the rational homology type of a space determines the rational higher Hochschild homology. In other words, if $X_*\to W_*$
    is a rational homology equivalence then the induced map $HH^{W_*}(\L_*)\to HH^{X_*}(\L_*)$ is an isomorphism. Similarly, the rational homology type of any map $X_*\to W_*$ determines the
    map in rational Hochschild-Pirashvili homology.  In particular, the rational homotopy type of a map must determine the higher Hochschild homology map. (In fact for suspensions the rational homology and rational homotopy equivalences are the same.)
    In this section we compute how exactly the map of suspensions induces the map of Hochschild complexes. For simplicity we will be assuming that the homology groups of the spaces that we consider are of finite type. Many of the results hold  without this restriction, but require more technical work involving careful colimit arguments. Since the goal is to make it applicable for concrete examples which in practice  always have this property, we concentrate on this case.
    
    \subsection{Determining the map of Koszul duals from the rational homotopy type of a map}\label{ss:determ1}
    First we need to understand how the map of Koszul duals
    \[
   \tilde H_* ((\Sigma Y_*)^{\wedge\bullet}/\Delta^\bullet \Sigma Y_*) \to
    \tilde H_* ((\Sigma Z_*)^{\wedge\bullet}/\Delta^\bullet \Sigma Z_*).
    \]
    is determined by the rational homotopy type of a map $f:\Sigma Y_*\to\Sigma Z_*$. 
    \sloppy
    Any such map  produces a commutative square of  right $\coLie\{1\}$-comodules:
    \begin{equation}\label{eq_sq}
\xymatrix{
  \tilde H_* \left((\Sigma Y_*)^{\wedge\bullet}/\Delta^\bullet Y_*\right)\ar[r]^-\simeq\ar[d] &\tilde H_*(\Sigma Y_*)^{\otimes\bullet}\circ\coLie\{1\}\ar[d] \\
   \tilde H_* \left((\Sigma Z_*)^{\wedge\bullet}/\Delta^\bullet Z_*\right)\ar[r]^-\simeq &\tilde H_*(\Sigma Z_*)^{\otimes\bullet}\circ\coLie\{1\} 
   }
  \end{equation}    
  The horizontal arrows are the isomorphisms from Corollary~\ref{cor:conf_susp}. We are interested in the right vertical map. (Notice that since $f$ is arbitrary and not necessarily a suspension, this right vertical map
  is not determined by the induced map in homology $f_*\colon\tilde H_*(\Sigma Y_*)\to\tilde H_*(\Sigma Z_*)$.)  According to Proposition~\ref{pr:left_act}, the horizontal maps respect the left
  $\Com$ action. It is quite obvious that the left vertical map does so as well. As a consequence, the right vertical map also respects this action.
  Its source is freely generated as a left $\Com$-module  by $\tilde H_*(\Sigma Y_*)^{\otimes 1}\circ\coLie\{1\}$, and its target is cofreely cogenerated as a $\coLie\{1\}$ right comodule by
  $\tilde H_*(\Sigma Z_*)^{\otimes\bullet}$.  As a consequence this map is determined by a map of symmetric sequences
  \[
  \tilde H_*(\Sigma Y_*)^{\otimes 1}\circ\coLie\{1\}\longrightarrow
  \tilde H_*(\Sigma Z_*)^{\otimes\bullet},
  \]
  or equivalently by a map
  \begin{equation}\label{eq_rht}
  \tilde H_*(Y_*)\to \FreeLie\left(\tilde H_*Z_*\right),
  \end{equation}
  where $\FreeLie\left(\tilde H_*Z_*\right)$ denotes the free completed Lie algebra generated by $\tilde H_*Z_*$.
  
  The rational homotopy of a simply connected suspension is a free Lie algebra generated by its reduced homology. We claim that in the simply connected case the map obtained in~\eqref{eq_rht} describes  exactly the map (of generators) of rational homotopy. More generally, when the suspensions are not necessarily simply connected, one can still assign a morphism~\eqref{eq_rht} to the rational homotopy type of a map $f\colon\Sigma Y_*\to\Sigma Z_*$. By Lemma~\ref{l:formality2} any suspension is rationally formal. Thus the induced map of their Sullivan\rq{}s models
  \[
  A_{\Sigma Z_*}\to A_{\Sigma Y_*}
  \]
  is quasi-isomorphic to a map of dg algebras
  \begin{equation}\label{eq:rat_model_map}
  \calA(\calL^c(\tilde H^*\Sigma Z_*))\to H^*\Sigma Y_*,
  \end{equation}
  where the left-hand side is the cofibrant replacement of $H^*\Sigma Z_*$ obtained as the Chevalley-Eilenberg complex $\calA(-)$ of the Harrison complex $\calL^c(-)$ of the (non-unital)
  algebra $\tilde H^*\Sigma Z_*$. Notice that $\calL^c(\tilde H^*\Sigma Z_*)$ is the cofree Lie coalgebra cogenerated by $\tilde H^*Z_*$  (with zero differential).  Its dual vector space is
  exactly $\FreeLie(\tilde H_*Z_*)$. The map of algebras~\eqref{eq:rat_model_map} is determined by its restriction on the space of generators
  \begin{equation}\label{eq:rat_model_gener}
   \calL^c(\tilde H^*\Sigma Z_*)\to \tilde H^* \Sigma Y_*.
   \end{equation}
   
   \begin{prop}\label{pr:thomas}
   For any map $f\colon\Sigma Y_*\to\Sigma Z_*$  of pointed suspensions of finite type, the map~\eqref{eq:rat_model_gener} encoding the rational homotopy type
   of $f$ is dual to the map~\eqref{eq_rht} encoding the homotopy type 
   of the induced map of right $\Com_+$ modules
   \begin{equation}\label{eq:cy_cz}
   \tilde C_*\left((\Sigma Y_*)^{\wedge\bullet}\right) \to \tilde C_*\left((\Sigma Z_*)^{\wedge\bullet}\right).
   \end{equation}
   \end{prop}
   
   \begin{proof}
   Arguing as in the proof of Lemma~\ref{l:formality1}, the map of right $\Com_+$ modules~\eqref{eq:cy_cz} is equivalent to the map
   \begin{equation}\label{eq:thomas1}
   (\tilde H_*\Sigma Y_*)^{\otimes \bullet}\to \left(\tilde\calA(\calL^c(\tilde H^*\Sigma Z_*))^{\otimes\bullet}\right)^\vee,
   \end{equation}
where $\tilde\calA(-)$ denotes the augmented part of $\calA(-)$; \lq\lq{}$^\vee$\rq\rq{} denotes taking the dual of a graded vector space. The map~\eqref{eq:thomas1} in each arity is the dual of a tensor power 
of~\eqref{eq:rat_model_map}.  The right-hand side of~\eqref{eq:thomas1} can also be expressed as $\left(\tilde \hChev(\FreeLie(\tilde H_* Z_*))\right)^{\hat\otimes\bullet}$, where $\tilde \hChev(-)$ denotes the completed augmented Chevalley-Eilenberg  complex (of a completed Lie algebra $\FreeLie(\tilde H_* Z_*)$); \lq\lq{}$\hat\otimes$\rq\rq{} denotes the completed tensor product.

One has a zigzag of right $\Com_+$-modules
\[
 (\tilde H_*\Sigma Y_*)^{\otimes \bullet}\to 
\left(\tilde \hChev (\FreeLie(\tilde H_* Z_*))\right)^{\hat\otimes\bullet}
 {\stackrel \simeq\longleftarrow}
 (\tilde H_*\Sigma Z_*)^{\otimes \bullet},
 \]
  where the right arrow is an equivalence. We get a zigzag of their Koszul duals:
\begin{equation}\label{eq:thomas2}
(\tilde H_*\Sigma Y_*)^{\otimes \bullet}\circ\coLie\{1\}\to 
\left(\tilde \hChev(\FreeLie(\tilde H_* Z_*))\right)^{\hat\otimes\bullet}\circ\coLie\{1\}{\stackrel \simeq\longleftarrow}
 (\tilde H_*\Sigma Z_*)^{\otimes \bullet}\circ\coLie\{1\},
 \end{equation}
 We claim that the right arrow has a natural left inverse. In order to construct this left inverse
 \[
 \left(\tilde \hChev(\FreeLie(\tilde H_* Z_*))\right)^{\hat\otimes\bullet}\circ\coLie\{1\}
  {\stackrel \simeq\longrightarrow}
 (\tilde H_*\Sigma Z_*)^{\otimes \bullet}\circ\coLie\{1\}
 \]
it is enough to define a map of their (co)generators
 \[
  \left(\tilde \hChev(\FreeLie(\tilde H_* Z_*))\right)^{\hat\otimes 1}\circ\coLie\{1\}
  \longrightarrow
 (\tilde H_*\Sigma Z_*)^{\otimes \bullet}.
 \] 
 In arity $n$ the latter map of symmetric sequences is defined as the following composition
 \begin{multline*}
  \tilde \hChev(\FreeLie(\tilde H_* Z_*))\otimes\coLie\{1\}(n)
   \to
  \FreeLie(\tilde H_* Z_*)[-1]\otimes\coLie\{1\} (n)\to \\
  \Lie(n)\otimes_{\bbS_n}(\tilde H_*\Sigma Z_*)^{\otimes n}\otimes \coLie(n) \to (\tilde H_*\Sigma Z_*)^{\otimes n}.
 \end{multline*}
 The first map is induced by the 
  projection on cogenerators $\tilde \hChev(\FreeLie(\tilde H_* Z_*))\to \FreeLie(\tilde H_* Z_*)[-1]$.
The second map is obtained by projecting
 $\FreeLie(\tilde H_*Z_*)$ onto its subspace spanned by brackets of length $n$. The last map takes into account the duality between the spaces $\Lie(n)$ and $\coLie(n)$:
 \[
L\otimes h_1\otimes\ldots \otimes h_n\otimes L' \mapsto \sum_{\sigma\in\bbS_n}(\sigma L,L') h_{\sigma_1}\otimes\ldots\otimes h_{\sigma_n}.
\]
 To finish the proof we notice that the composite of the first arrow in~\eqref{eq:thomas2}  and the constructed inverse is the map
 \[
 (\tilde H_*\Sigma Y_*)^{\otimes \bullet}\circ\coLie\{1\}\to 
 (\tilde H_*\Sigma Z_*)^{\otimes \bullet}\circ\coLie\{1\}
 \]
 (co)generated by  the map dual do~\eqref{eq:rat_model_gener}.
 \end{proof}
 
\subsection{Determining map of Hochschild-Pirashvili homology}\label{ss:determ2}
In this subsection we describe how the map
\begin{equation}\label{eq:rhmap}
\tilde H_*Y_*\to \FreeLie(\tilde H_*Z)
\end{equation}
encoding the rational homotopy type of $f\colon\Sigma Y_*\to\Sigma Z_*$, determines the map of higher Hochschild complexes $\CH^{(-)}(-)$ (in fact we will work with $\gr\, \CH^{(-)}(-)$ instead).
For simplicity we will be assuming that $Y_*$ and $Z_*$ are of finite type and we will only look at the case $\L_*=M\otimes \Chev(\alg g)^{\otimes \bullet}$, where $\alg g$ is strictly positively graded. Thus we need to describe the induced map
\begin{equation}\label{eq:expl1}
M\,\hat\otimes\, S\left(\tilde H^*Z_*\hat\otimes\alg g\right)\to M\,\hat\otimes \,S\left(\tilde H^*Y_*\hat\otimes\alg g\right).
\end{equation}
Firstly, this map is the tensor product of the identity on the first factor $M$ and a coalgebra homomorphism on the second one. Ergo, it\rq{}s enough to describe its composition with the projection to the space of cogenerators
\begin{equation}\label{eq:CH_map_cogener}
 S\left(\tilde H^*Z_*\hat\otimes\alg g\right)\to \tilde H^*Y_*\hat\otimes\alg g.
\end{equation}
The map~\eqref{eq:rhmap} is a product of maps
\begin{equation}\label{eq:rhmap_n}
\tilde H_*Y\to \Lie(n)\otimes_{\bbS_n}(\tilde H_* Z_*)^{\otimes n}.
\end{equation}
its $n$-th component~\eqref{eq:rhmap_n} can be viewed as an element $\rho_n\in\tilde H^*Y_*\hat\otimes\Lie(n)\otimes_{\bbS_n}(\tilde H_*Z)^{\otimes n}$. This element $\rho_n$ 
contributes only to
\begin{equation}\label{eq:CH_map_cogener_n}
 S^n\left(\tilde H^*Z_*\otimes\alg g\right)\to \tilde H^*Y_*\otimes\alg g.
\end{equation}
in~\eqref{eq:CH_map_cogener}.  The element $\rho_n$ is a sum of elements of the form
 \[
 h^0\otimes L\otimes h_1\otimes\ldots \otimes h_n\in\tilde H^*Y_*\otimes\Lie(n)\otimes_{\bbS_n}(\tilde H_*Z)^{\otimes n}.
 \]
 Each such summand contributes to~\eqref{eq:CH_map_cogener_n} as a map sending
 \[
 (h^1\otimes g_1)\cdot \ldots\cdot (h^n\otimes g_n)\in S^n\left(\tilde H^*Z_*\otimes\alg g\right)
 \]
 to
 \[
 \sum_{\sigma\in\bbS_n}\pm\left(\prod_{i=1}^m (h_i,h^{\sigma_i})\right) h^0\otimes L(g_{\sigma_1},\ldots,g_{\sigma_n}) \in \tilde H^*Y_*\hat\otimes\alg g,
 \]
 where the sign is as usual the Koszul one induced by permutation of elements.
 
 In the examples below we will be omiting the hat sign over the tensor product as the induced map~\eqref{eq:expl1} can always be restricted on the non-completed part $M\otimes S(\tilde H^*(-)\otimes \alg g)$ (where 
 the symmetric power is also taken in the non-completed sense.)

\begin{ex}
Consider the map $S^1\to S^1\vee S^1$ which sends the generator $x$ of $\pi_1S^1$ to the product $y_1y_2$ of generators of $\pi_1( S^1\vee S^1)$. The map~\eqref{eq:rhmap} becomes
\[
x\Q\to \FreeLie(y_1,y_2),
\]
that encodes the map of the primitive part of the Malcev completions~\cite{FHT2} (all generators $x$, $y_1$, $y_2$ are of degree zero). The image of $x$ is described by the 
Baker-Campbell-Hausdorff formula
\[
x\mapsto \ln(e^{y_1}\cdot e^{y_2}).
\]
The map~\eqref{eq:expl1} becomes
\[
M\otimes S(\alg g)\otimes S(\alg g)\to M\otimes S(\alg g)
\]
which sends
\[
m\otimes A\otimes B\mapsto m\otimes A\star B,
\]
where $\star$ is the associative (star) product on $S(\alg g)$ transported from $\U\alg g$ via the Poincar\'e-Birkhoff-Witt isomorphism.
\end{ex}

\begin{ex}
Consider the map $S^2\to S^1\vee S^2$ corresponding to the element $x\cdot y\in \pi_2 ( S^1\vee S^2)$, where $x$ is the generator of $\pi_1 S^1$ and $y$ is the generator of $\pi_2 S^2$.
The map~\eqref{eq:rhmap} in our case is
\[
y\Q\to \FreeLie(x,y),
\]
where $|x|=0$, $|y|=1$,
\[
y\mapsto e^{ad_x}(y).
\]
The induced map~\eqref{eq:expl1}  is
\[
M\otimes S(\alg g)\otimes S(\alg g[1])\to M\otimes S(\alg g[1]),
\]
sending
\[
m\otimes g_1\cdot\ldots \cdot g_k\otimes s^{-1}g'_1\cdot\ldots\cdot s^{-1}g'_{k'}
\mapsto
m\otimes\frac 1{k!}\sum_{\sigma\in\bbS_k} ad_{g_{\sigma_1}}\ldots ad_{g_{\sigma_k}}( s^{-1}g'_1\cdot\ldots\cdot s^{-1}g'_{k'}).
\]
\end{ex}

\begin{ex}
Consider the Hopf map $S^3\to S^2$. On the level of rational homotopy we get a map
\[
y\Q\to \FreeLie(x),
\]
where $|x|=1$, $|y|=2$, and
\[
y\mapsto \frac 12 [x,x].
\]
The induced map of higher Hochschild complexes
\[
M\otimes S(\alg g[1])\to M\otimes S(\alg g[2])
\]
sends
\begin{gather*}
m\otimes s^{-1}g_1\cdot\ldots \cdot s^{-1}g_{2k-1}\mapsto 0,\\
m\otimes s^{-1}g_1\cdot\ldots \cdot s^{-1}g_{2k}\mapsto m\otimes \frac{1}{2^k k!}\sum_{\sigma\in\bbS_{2k}} \pm s^{-2}[g_{\sigma_1},g_{\sigma_2}]\cdot\ldots\cdot s^{-2}[g_{\sigma_{2k-1}},g_{\sigma_{2k}}]. 
\end{gather*}

\end{ex}

\section{Hochschild-Pirashvili homology for non-suspensions}\label{s:non_susp}
Some of the techniques given in the present paper can also be applied to study the higher Hochschild homology for non-suspensions and maps between them. This section is a short note on how this works in the special case when $\L_*=M\otimes \Chev(\alg g)^{\otimes\bullet}$, where $\alg g$ as usual is a strictly positively graded dg Lie algebra, and the spaces are connected and of finite type.

\begin{thm}\label{th:non_susp}
Assuming a pointed space $X_*$ is connected and of finite type, let $A$ be an augmented non-positively graded augmented commutative dg algebra of finite type quasi-isomorphic to the Sullivan algebra $A_{X_*}$, and $\tilde A$ be its augmentation ideal.\footnote{In our conventions all the complexes have differential of degree $-1$, for which reason the algebras we consider are  non-positively graded.}  Then the Hochschild-Pirashvili homology $HH^{X_*}(\Chev(\alg g),M)$ is computed by the complex $M\hat\otimes \hChev(\tilde A\hat\otimes\alg g)$, where 
$\hChev(\tilde A\hat\otimes\alg g)$ is the completed (with respect to the total homological degree of
elements from $\alg g$) Chevalley-Eilenberg complex of the completed Lie algebra $\tilde A\hat\otimes\alg g$. The differential
has the form
\begin{equation}\label{eq:differential_any_space}
d=d_M+d_{\alg g}+d_A+d_{CE}+\delta,
\end{equation}
where $d_M$, $d_{\alg g}$, $\delta$ are as those from~\eqref{eq:differential2}, $d_A$ is induced by the differential in $A$, $d_{CE}$ is the Chevalley-Eilenberg differential.
\end{thm}

\begin{proof}
This complex is constructed in the same way as the higher Hochschild complexes for suspensions, see Subsection~\ref{ss41}. The extra term  $d_{CE}$ in the differential appears due to the fact that the $\Com_+$ action on $(\tilde A^\vee)^{\otimes \bullet}$ is now non-trivial.
\end{proof}

The result of this theorem is partially known to experts. It appeared explicitly for spheres and surfaces respectively in~\cite[Theorem~3]{Ginot} and \cite[Theorem~4.3.3]{GTZ0}, 
see also~\cite{AyFr} for a similar implicit statement in case
$X$ is a manifold. Notice also that in case $M=\Chev(\alg g)$ (i.e., when considering unpointed version of higher Hochschild homology) the obtained higher 
Hochschild complex is the completed Chevalley-Eilenberg complex $\hChev(A\hat \otimes \alg g)$. 
As application of this example, in case the dimension of $X$ is less than the connectivity of $Y$, the space $Y^X$ of continuous maps $Y\to X$ has homology with any coefficients described as 
$H_*(Y^X)\simeq HH^X(C_*(Y^\bullet))$, see~\cite{PatrasThomas,pirash00}. On the other hand,
the rational homotopy type of $Y^X$ is described by the dg Lie algebra $A\hat\otimes L$, where $A$ is a suitable Sullivan model for $X$ and $L$ is a suitabe 
Quillen model for $Y$, see~\cite{BlLaz,BrSz,BuFeMu}.
From this we also recover that $\hChev(A\hat\otimes L)$, i.e., our complex, computes the rational homology of $Y^X$.

\begin{rem}\label{r:non_susp_Hodge}
One can easily see that the $m$th term of the Hodge filtration in 
$
M\hat\otimes \hChev(\tilde A\hat\otimes\alg g) = \prod_{i=0}^{+\infty} M\hat\otimes S^i (\tilde A[-1]
\hat\otimes \alg g)
$
is 
$
F_m M\hat\otimes \hChev(\tilde A\hat\otimes\alg g) = \prod_{i=0}^{m} M\hat\otimes S^i (\tilde A[-1]
\hat\otimes \alg g).
$
\end{rem}

Theorem~\ref{th:non_susp} applied to a suspension $\Sigma Y_*$ of a finite type is exactly the statement of
Remark~\ref{r:grCH_C_M}. Indeed, since $\Sigma Y_*$ is formal one can take $\tilde A=\tilde H^*\Sigma Y_*$ the cohomology algebra, whose product is trivial, and thus the Chevalley-Eilenberg part of the differential  is trivial  $d_{CE}=0$.  The rational homotopy type of a map of suspensions of finite type 
$f\colon\Sigma Y_*\to \colon \Sigma Z_*$ is encoded by a map~\eqref{eq:rat_model_gener}, which is essentially the same as a $\Com_\infty$ map of commutative algebras 
$f^*_\infty\colon \tilde H^*\Sigma Z_*\to \tilde H^*\Sigma Y_*$. In Subsection~\ref{ss:determ2} we show how this map determines a map of higher Hochschild complexes
\[
M\hat\otimes\hChev(\tilde H^*\Sigma Z_*\hat\otimes\alg g)\to
M\hat\otimes\hChev(\tilde H^*\Sigma Y_*\hat\otimes\alg g),
\]
which is the identity on the first factor $M$ and a completed coalgebras map on the second factor. The latter 
map can be regarded as a completed $L_\infty$ morphism 
\[
 \tilde H^*\Sigma Z_*\hat\otimes {\alg g}\to \tilde H^*\Sigma Y_*\hat\otimes{\alg g}.
 \]
 of (completed) abelian Lie algebras.
 
 More generally, a tensor product with a dg Lie algebra is in fact a functor from $\Com_\infty$ algebras to $\L_\infty$ algebras. We will need a completed version of this construction.  Let $\tilde A$ be a negatively graded $\Com_\infty$ algebra of finite type encoding the rational homotopy type of a connected pointed 
 space $X_*$, and let $\alg g$ be a positively graded dg Lie algebra. The completed $L_\infty$ algebra
 structure on $\tilde A\hat \otimes \alg g$ is explicitly described by the structure maps $\mu_n$ defined as 
 composition
 \begin{equation}\label{eq_A_infty}
 \mu_n\colon S^n(\tilde A[-1]\hat\otimes {\alg g} )
 \to \FreeLie^c(\tilde A[-1])\hat\otimes \FreeLie(\alg g)
 \to \tilde A\hat \otimes \alg g,
 \end{equation}
 where $\FreeLie^c(\tilde A[-1])$ is the free Lie coalgebra cogenerated by $A[-1]$ (in other words, 
 it is the Harrison complex $\calL^c(\tilde A)$). The first map is induced by the diagonal
 $\Com(n)\to \coLie(n)\otimes \Lie(n)$. The second map is the $\Com_\infty$ structure on the first factor
 and the $\Lie$ structure on the second. If $\tilde B\to\tilde A$ is a $\Com_\infty$ morphism encoding the rational homotopy type of a pointed map $X_*\to Y_*$, then the induced completed $L_\infty$ map
 $\tilde B\hat\otimes \alg g\to \tilde A\hat\otimes \alg g$ is described by essentially the same formulas
 as~\eqref{eq_A_infty}. Its $n$-th component is the composition
 \begin{equation}\label{eq_A_infty2}
 F_n\colon S^n(\tilde B[-1]\hat\otimes {\alg g} )
 \to \FreeLie^c(\tilde B[-1])\hat\otimes \FreeLie(\alg g)
 \to \tilde A[-1]\hat \otimes \alg g,
 \end{equation}
 where the first map is the same as the fist one in~\eqref{eq_A_infty}. The second map is the tensor product of the $\Com_\infty$ map $\tilde B\to \tilde A$ and the Lie algebra structure map on $\alg g$.
 In  Subsection~\ref{ss:determ2} the corresponding $L_\infty$ map is explained in full detail for the case of
 suspensions $\tilde A=\tilde H^*\Sigma Y_*$, $\tilde B=\tilde H^*\Sigma Z_*$.

\begin{rem}\label{r:l_infty}
For a $\Com_\infty$ algebra $\tilde A$  (non-positively graded and of finite type) consider its dual $\Com_\infty$ coalgebra $\tilde A^\vee$. Then
the $L_\infty$ algebra  $\tilde A\hat\otimes \alg g$ considered above is the $L_\infty$ algebra of derivations
of the zero map of Lie algebras $\calL(\tilde A^\vee)\to \alg g$. 
\end{rem}

\begin{thm}\label{th:non_susp2}
Let $\tilde A$ be a non-positively graded $\Com_\infty$  algebra of finite type encoding the rational homotopy type
of a pointed space $X_*$, then the Hochschild-Pirashvili homology $HH^{X_*}(\Chev(\alg g),M)$ is 
computed by the complex $M\hat\otimes \hChev(\tilde A\otimes\alg g)$, where 
 $\hChev(\tilde A\otimes\alg g)$ is the completed Chevalley-Eilenberg complex of the completed
 $L_\infty$ algebra $\tilde A\hat \otimes \alg g$. The differential has the form~\eqref{eq:differential_any_space}. If $\tilde B\to \tilde A$ is a $\Com_\infty$ morphism 
 (of  non-positively graded $\Com_\infty$ algebras of finite type) encoding the rational homotopy type of a pointed map
 $X_*\to Y_*$, then the induced map in the Hochschild-Pirashvili homology 
\[
HH^{Y_*}(\Chev(\alg g),M)\to HH^{X_*}(\Chev(\alg g),M)
\]
 is computed by the chain map
 \[
 M\hat\otimes \hChev(\tilde B\otimes\alg g) \to M\hat\otimes \hChev(\tilde A\otimes\alg g),
 \]
 which is identity on the first factor $M$ and a completed coalgebra map corresponding to the induced
 completed $L_\infty$ algebras map $\tilde B\hat \otimes \alg g \to \tilde A\hat \otimes \alg g$. 
\end{thm}

\begin{proof}
First we check that the statement  of the theorem holds when $\tilde B\to\tilde A$ is a dg commutative algebras map, which is an easy refinement of Theorem~\ref{th:non_susp}. On the other hand hand 
any $\Com_\infty$ algebra (and any $\Com_\infty$ morphism) is quasi-isomorphic to a dg commutative algebra (map of dg commutative algebras). This together with the fact that a $\Com_\infty$ quasi-isomorphism $\tilde A_1\to \tilde A_2$ induces an $L_\infty$ 
quasi-isomorphism $\tilde A_1\hat\otimes \alg g\to \tilde A_2\hat\otimes \alg g$ proves the staement of the theorem.
\end{proof}

The above theorem has the following corollary.

\begin{prop}\label{p:non_susp_no}
For a pointed connected space $X_*$ of finite type, the Hodge filtration  in the higher Hochschild complexes splits for any coefficient $\Gamma$ module $\L_*$ if and only if $X_*$ is rationally homology equivalent to a suspension. 
\end{prop}

\begin{proof}
In one direction the statement easily follows from the fact that a rational homology equivalence of spaces induces a quasi-isomorphism of higher Hochschild complexes. Now let $X_*$ be not equivalent to a suspension. It is well known that any $\Com_\infty$ algebra is $\Com_\infty$ quasi-isomorphic to a  one with zero differential \cite[Theorem 10.4.5]{lodayval}. Let $\tilde A$ be such one encoding the rational homotopy type of $X_*$.
Since we assume $X_*$ is not rationally a suspension, $\tilde A$ must have non-trivial (higher) product(s).
Let $k$ be the arity of the first non-trivial product. We choose 
$\L_*=M\otimes\Chev(\alg g)^{\otimes\bullet}$, where $M=\Q$ is the  comodule with the trivial coaction,
and $\alg g$ is a free Lie algebra with $k$ generators. By construction $\tilde A\hat\otimes\alg g$
is an $L_\infty$ algebra with zero differential and whose first non-trivial (higher) bracket has arity $k$. 
Applying  Remark~\ref{r:non_susp_Hodge} we get that the $(k-1)$th differential in the spectral sequence associated
with the Hodge filtration in $ M\hat\otimes \hChev(\tilde A\otimes\alg g)$ is non-zero. Therefore 
the filtration does not split.
\end{proof}

\subsection{Hochschild-Pirashvili homology as "homotopy base change"}
Let us conclude by remarking on a curious algebraic interpretation of the Hochschild-Pirashvili homology in the form described in Theorem \ref{th:non_susp}. First, recall that to any dg commutative algebra $A$ we may associate a functor 
\[
\Phi_A : (\text{Lie algebras}) \to (\text{Lie algebras})
\]
by sending a dg Lie algebra $\alg g$ to the tensor product $\Phi_A(\alg g) := \alg g\otimes A$, with the Lie algebra structure $A$-linearly extended in the obvious manner. We may call this functor $\Phi_A$ "base change", even though this is a misnomer as we do not change the underlying ground ring.
Similarly, if $\alg g$ is a dg Lie algebra and $K$ is an $A$-module, we may define a functor 
\[
\Psi_{A,K} : (\alg g-\text{modules}) \to (\Phi_A(\alg g)-\text{modules})
\] 
by sending a $\alg g$-module $\alg k$ to the $\Phi_A(\alg g)$-module $\Psi_{A,K}(\alg k):=\alg k \otimes K$, with the module structure defined in the obvious manner. We also call the functor $\Psi_{A,K}$ "base change", with the same caveat as above that this is a misnomer.
There is also a topological variant: If the Lie algebra $\alg g$ carries in addition a complete topology compatible with the Lie algebra structure, then $\hat \Phi_A(\alg g):=  \alg g \hat\otimes A$ is likewise equipped with a natural complete filtration. Similarly, if $\alg k$ is equipped with a complete filtration and the action of $A$ is continuous, then $\hat\Psi_{A,K}(\alg k):=\alg k \hat \otimes K$ is a complete (continuous) $\Phi_A(\alg g)$-module.

Now it is well known \cite[chapter 11.3]{lodayval} that there is an adjunction of categories
\[
\Ha : (\text{conilpotent coaugmented dg cocommutative coalgebras}) \leftrightarrows (\text{dg Lie algebras}): \Chev
\]
given by the bar and cobar functors (i.e., the Harrison and Chevalley complex functors), such that for any conilpotent dg coalgebra $C$ the unit of the adjunction $C\to \Chev(\Ha (C))$ is a quasi-isomorphism, and such that for any dg Lie algebra $\alg g$ the counit of the adjunction $\Ha(\Chev(\alg g))\to \alg g$ is a quasi-isomorphism.
Concretely, the functor $\Ha$ takes the Harrison complex (a free Lie algebra) of the cokernel of the coaugmentation, while the functor $\Chev$ takes the Chevalley complex.
Similar functors exist on the the level of comodules. If $C$ is a conilpotent dg cocommutative coalgebra then we have bar and cobar functors 
\begin{align*}
\HaMod : (\text{conilpotent $C$-comodules}) \to (\Ha (C)-\text{modules})  \\
 (\text{conilpotent $\Chev(\Ha (C))$-comodules}) \leftarrow (\Ha (C)-\text{modules}) : \ChevMod .
\end{align*}
Concretely, $\calL_{\rm mod}(M)=\Harr(C;M)$ is the Harrison complex with values in the module $M$, i.e., a free $\Ha (C)$-module generated by $M$ if we disregard the differential. Similarly, $\ChevMod(N)=\Chev(\Ha (C);N)$ is the Chevalley complex with values in $N$, i.e., a cofree $\Chev(\Ha (C))$-comodule cogenerated by $N$ with a natural differential.

There exist versions of the above constructions for complete topological algebras and modules, by replacing tensor products appearing there by a completed version. We denote those completed versions by $\hHa$, $\hChev$ etc. 

The above adjunctions allow us to transport any endofunctor of the category of dg Lie algebras to an endofunctor of the category of conilpotent dg cocommutative coalgebras (and vice versa). The point of this section is to remark that the Hochschild-Pirashvili homology functor is nothing but the (homology of the) well known base change functors above, transported to the category of conilpotent coalgebras via the bar and cobar adjunctions. This gives an algebraically "very simple" interpretation of the Hochschild-Pirashvili homology.
Concretely, let us assume that we are given the following data:
\begin{itemize}
\item A conilpotent complete cocommutative dg coalgebra $C$, for example $C=\Chev(\alg g)$, for 
a dg Lie algebra $\alg g$ as in Theorem~\ref{th:non_susp}, which we endow with the complete  the
decreasing  filtration by degree.
\item A a conilpotent complete $C$-comodule $M$.
\item An augmented dg commutative algebra $A$. For example, we may take such an $A$ from Theorem~\ref{th:non_susp}. We will still denote by $\tilde A$ its augmentation ideal.
\item We let $K=\mathbb{Q}$  be the one-dimensional $A$-module, with the action defined by the augmentation.
\end{itemize}

Then we define a complete cocommutative coalgebra 
\[
C_A := \hChev(\hat\Phi_A(\Ha(C)) ) = \hChev(\Ha(C)\hat \otimes A ) 
\]
and the complete $C_A$-comodule
\[
M_A := \hChevMod(\hat\Psi_{A,K}(\HaMod(M)) ) = \hChevMod(\HaMod(M)\hat \otimes K ). 
\]
Clearly, these constructions are functorial in $A$, $C$ and $M$. We will abusively call these constructions "homotopy base change".
The main statement of this section is then that the complex of Theorem \ref{th:non_susp} computing the Hochschild-Pirashvili homology may be interpreted as "homotopy base change".
\begin{prop}\label{p:non_susp}
For $C=\Chev(\alg g)$ the Chevalley complex of a dg Lie algebra, and $A,M$ as above, the complexes $M_A$ and the complex $( M \hat \otimes \hChev(\tilde A \hat\otimes \alg g),d)$ of Theorem \ref{th:non_susp} are quasi-isomorphic. 
\end{prop}
\begin{proof}
Explicitly, the complex $M_A$ has the form 
\[
\hChev(\Ha(C)\hat \otimes A;  \Harr(C; M)\otimes K )
\]
where $\hChev(-;-)$ denotes the (completed) Chevalley complex with values in the second argument, and $\Harr(-; -)$ denotes the Harrison complex. Using the augmentation we may now split $A=\mathbb{Q}\oplus \tilde A$, where $\tilde A$ is the kernel of the augmentation. Using this splitting we find the identification of graded vector spaces (recall that $K=\mathbb{Q}$)
 \begin{equation}\label{equ:Cisomorphism}
\hChev(\Ha(C)\hat\otimes A;  \Harr(C; M)\otimes K )
\cong 
\hChev(\Ha(C)\hat \otimes \tilde A)\hat \otimes \Chev(\Ha(C);  \Harr(C; M) ).
\end{equation}
Note however, that this identification is not an identification of complexes (yet). The differential on the right-hand side is composed of two terms: the differential $d_1$ of the left-hand tensor factor and the differential $d_2$ of the right-hand tensor factor. The differential on the left-hand side of \eqref{equ:Cisomorphism} on the other hand has an additional term $d_{mixed}$ from the Chevalley differential, which is obtained by taking the coaction of $\Chev(\Ha(C);  \Harr(C; M) )$ followed by a Lie bracket. Note that this term resembles the term $\delta$ in Theorem \ref{th:non_susp}.
Note that we have a quasi-isomorphism of $\Chev(\Ha( C))$-comodules
\[
M \to \Chev(\Ha(C);  \Harr(C; M) ).
\]
Hence we obtain a quasi-isomorphism of complexes
 \begin{equation}\label{equ:Cisomorphism1}
(\hChev(\Ha(C)\hat\otimes \tilde A) \otimes M, d_1+d_M+d_{mixed})
\stackrel{\sim}{\to}
(\hChev(\Ha(C)\hat\otimes \tilde A) \hat\otimes \Chev(\Ha(C);  \Harr(C; M) ), d_1+d_2+d_{mixed})
\cong M_A
,
\end{equation}
where the part of the differential $d_{mixed}$ on the left-hand complex is defined as before by taking the coaction on $M$ followed by a Lie bracket with a factor of $\hChev(\Ha(C)\hat\otimes \tilde A)$.

Furthermore, since $C=\Chev(\alg g)$ we have a quasi-isomorphism of dg Lie algebras
\[
\Ha(C) \to \alg g.
\]
Hence we obtain a quasi-isomorphism
 \begin{equation}\label{equ:Cisomorphism2}
(\hChev(\Ha(C)\hat\otimes \tilde A)\hat \otimes M, d_1+d_M+d_{mixed})
\stackrel{\sim}{\to}
(\hChev(\alg g\hat \otimes \tilde A) \hat\otimes M, d)
\end{equation}
with the complex considered in Theorem \ref{th:non_susp}.
By \eqref{equ:Cisomorphism1} and \eqref{equ:Cisomorphism2} the Proposition is shown.
\end{proof}

\bibliographystyle{plain}

\end{document}